\documentclass[10pt]{amsart}
\allowdisplaybreaks

\usepackage{amsfonts}
\usepackage{amsmath}
\usepackage{amssymb}
\usepackage{amsthm}
\usepackage{amsbsy}
\usepackage{xcolor}
\usepackage{subcaption}
\usepackage{graphicx} \usepackage{enumerate} \usepackage{multicol}
\usepackage{mathrsfs} \usepackage[all,cmtip]{xy}
\usepackage[color=blue!20]{todonotes}
\usepackage[ruled]{algorithm2e}
\usepackage{url}

\newcounter{dummy} \numberwithin{dummy}{section}
\newtheorem{theorem}[dummy]{Theorem}
\newtheorem{corollary}[dummy]{Corollary}

\newtheorem{definition}[dummy]{Definition}
\newtheorem{proposition}[dummy]{Proposition}
\theoremstyle{remark}
\newtheorem{remark}[dummy]{Remark}
\newtheorem{example}[dummy]{Example}

\newcommand{\calH}{\mathcal{H}}
\newcommand{\calV}{\mathcal{V}}

\newcommand{\calN}{\mathcal{N}}

\newcommand\E{\mathbb{E}}

\DeclareMathOperator{\Ann}{Ann}
\DeclareMathOperator{\Sym}{Sym}

\DeclareMathOperator{\pr}{pr}
\DeclareMathOperator{\rank}{rank}
\DeclareMathOperator{\spn}{span}
\DeclareMathOperator{\tr}{tr}

\newcommand{\ptr}{/ \! /}

\DeclareMathOperator{\Ad}{Ad}

\DeclareMathOperator{\GL}{GL}
\DeclareMathOperator{\gl}{\mathfrak{gl}}
\DeclareMathOperator{\SO}{SO}
\DeclareMathOperator{\Ort}{O}
\DeclareMathOperator{\so}{\mathfrak{so}}

\DeclareMathOperator*{\argmin}{argmin}

\DeclareMathOperator{\vl}{vl}
\DeclareMathOperator{\sn}{sn}

\DeclareMathOperator{\dn}{dn}
\DeclareMathOperator{\diag}{diag}

\newcommand{\ve}{\varepsilon}

\numberwithin{equation}{section}

\title[
Most probable paths for anisotropic Brownian motions]{Most probable paths for anisotropic Brownian motions on manifolds
}
\author[E.~Grong and S.~Sommer]{Erlend Grong and Stefan Sommer}

\address{University of Bergen, Department of Mathematics, P.O.~Box 7803, 5020 Bergen, Norway}
\email{erlend.grong@uib.com}
  
\address{University of Copenhagen, Universitetsparken 5, CK-2100 Copenhagen E, Denmark}
\email{sommer@di.ku.dk}

\thanks{The first author is supported by the grant GeoProCo from the Trond Mohn Foundation - Grant TMS2021STG02 (GeoProCo). The second author is supported by the Villum Foundation grant 
00022924, and the Novo Nordisk Foundation grant NNF18OC0052000.}


\begin{document}

\begin{abstract}
Brownian motion on manifolds with non-trivial diffusion coefficient can be constructed by stochastic development of Euclidean Brownian motions using the fiber bundle of linear frames. We provide a comprehensive study of paths for such processes that are most probable in the sense of Onsager-Machlup, however with path probability measured on the driving Euclidean processes. 
We obtain both a full characterization of the resulting family of most probable paths, reduced equation systems for the path dynamics where the effect of curvature is directly identifiable, and explicit equations in special cases, including constant curvature surfaces where the coupling between curvature and covariance can be explicitly identified in the dynamics. We show how the resulting systems can be integrated numerically and use this to provide examples of most probable paths on different geometries and new algorithms for estimation of mean and infinitesimal covariance.
\end{abstract}

\maketitle

\section{Introduction}
The Eells-Elworthy-Malliavin \cite{malliavin_stochastic_1976} construction of Brownian motion applies stochastic development to map Euclidean Brownian motions to Riemannian manifolds. The construction uses a Stratonovich SDE in the orthonormal frame bundle to generate the manifold valued processes. A more general class of processes can be constructed by relaxing the requirement of the SDE to start with an orthonormal frame. This corresponds to choosing a diffusion coefficient with non-trivial covariance between the infinitesimal steps of the process. Such processes have been studied in geometric statistics where they, for example, are used to generate probability distributions on manifolds that can be interpreted as normal distributions \cite{sommer_anisotropic_2015,sommer_modelling_2017}. These distributions and the corresponding generating processes are termed anisotropic due to their directionally dependent diffusion, and the probability of observing a given point on the manifold is influenced by this anisotropy which weighs the probability of paths from the starting point to the data.

The most probable paths for a Brownian motion from its starting point to a fixed end point can be described as extremal values of the Onsager-Machlup functional \cite{fujita_onsager-machlup_1982}. This notion is generalized to the anisotropic case in \cite{sommer_modelling_2017} by measuring the path probability on the driving process, i.e. the Euclidean Brownian motion that is mapped by stochastic development to the manifold. The resulting family of paths are geodesics for a sub-Riemannian structure on the frame bundle \cite{sommer_anisotropically_2016}. In this paper, we give an in-depth study of the family of most probable paths. We fully characterize the family and show that the family is a subset of normal sub-Riemannian geodesics. We derive a reduced system governing the dynamics of the family and make the influence of curvature on the dynamics explicit. We link this to qualitative aspects of the family, particularly showing how the paths bend towards directions of high-variance with positive curvature, and low-variance with negative curvature, when compared to a Riemannian geodesic connecting their endpoints. We furthermore derive new algorithms for mean and infinitesimal covariance estimation using the paths, and show how efficient optimization can be obtained using both properties of the systems - particularly equivariance to scaling of the covariance - and using automatic differentiation as opposed to direct evaluation of adjoint equations.
\begin{figure}[h]
    \centering
    \begin{subfigure}[b]{\textwidth}
        \centering
        \includegraphics[trim=200 240 200 200,clip,width=0.35\linewidth]{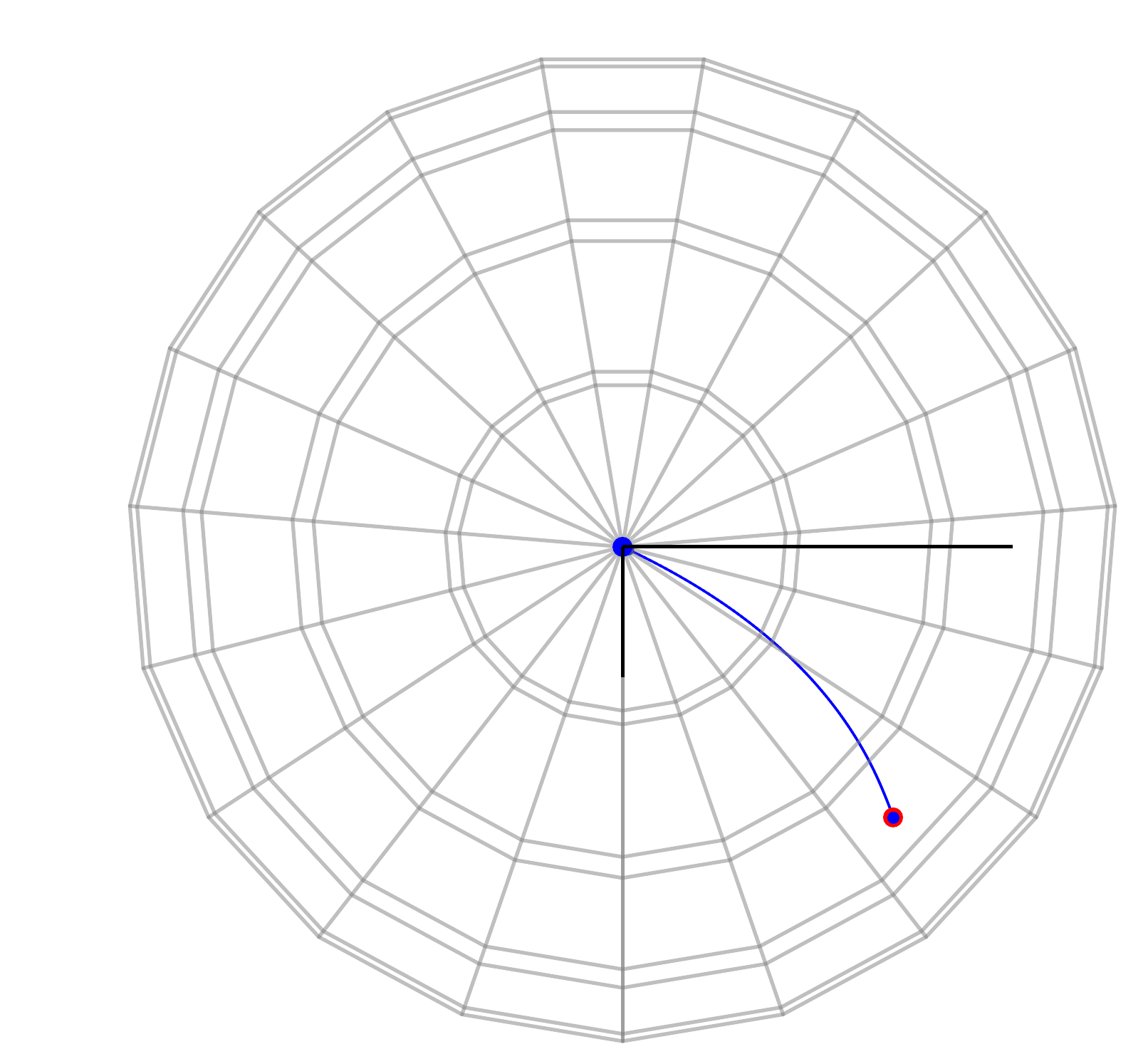}
        \hspace{1cm}
        \includegraphics[trim=200 280 200 200,clip,width=0.35\linewidth]{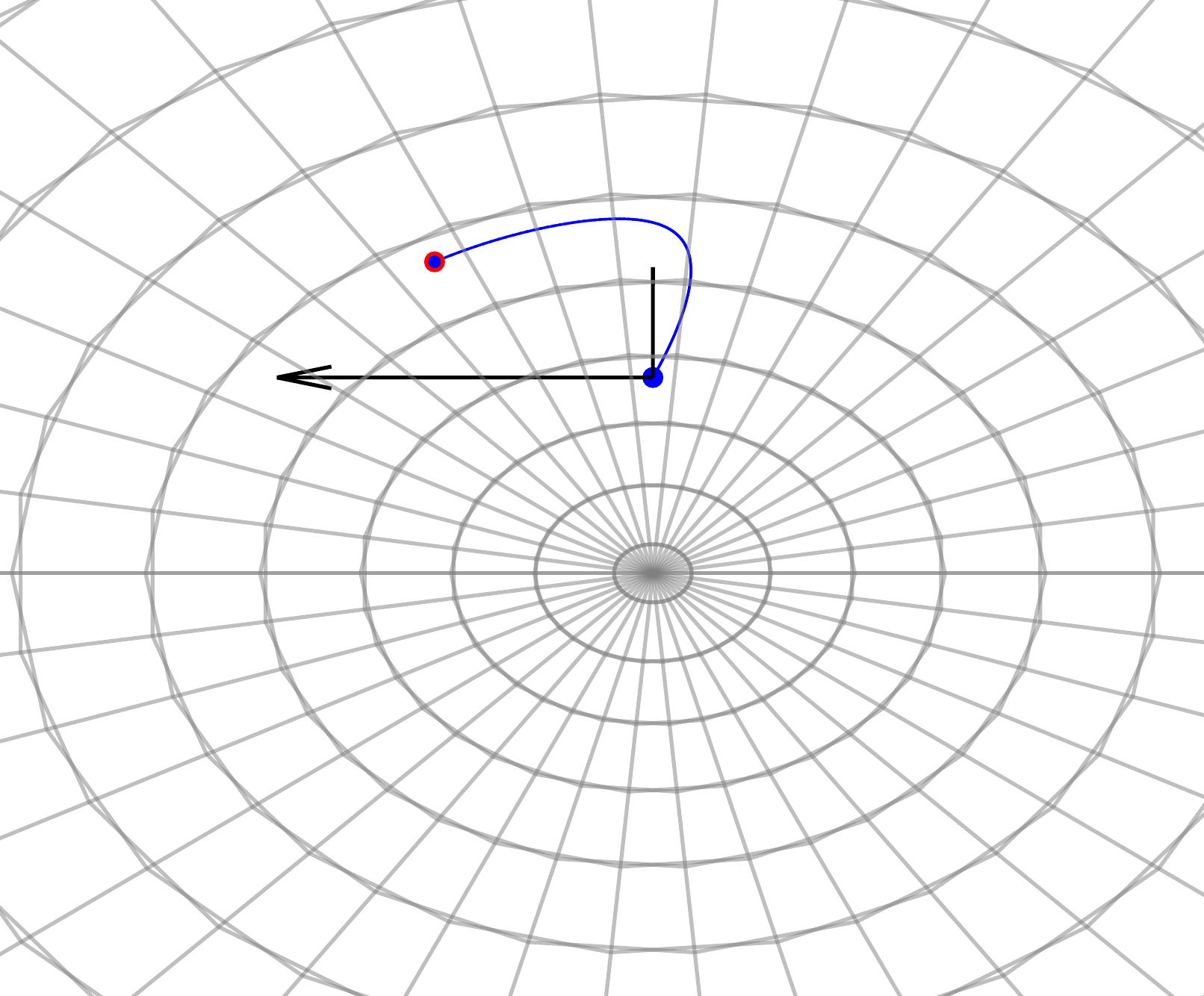}
    \end{subfigure}
    \caption{Examples of most probable paths on the sphere $S^2$ (left, view from a pole) and the hyperbolic space $H^2$ (right). The corresponding frames encoding the diffusion covariance is shown above the starting points (black lines, longer implies higher covariance). Note that compared to geodesics connecting endpoint, the most probable paths bend towards the large-covariance direction with positive curvature and to the low-covariance direction with negative curvature.}
    \label{fig:intro}
\end{figure}

\subsection{Background}
Most probable paths for anisotropic diffusion processes were defined in \cite{sommer_modelling_2017} where the Onsager-Machlup \cite{fujita_onsager-machlup_1982} functional was used to characterize the path probability from the Euclidean Brownian motion that drives the evolution. The paths were further studied in \cite{sommer_evolution_2015} as projections of sub-Riemannian geodesics on the frame bundle $FM$ with the assumption of normality. The resulting Hamiltonian system provides only sufficient conditions for the paths, and it has not been proved that the normality assumption holds. Results in \cite{sommer_evolution_2015} show experimentally that the anisotropy affects the dynamics of the paths, meaning that they travel more in directions of largest variance than geodesic connecting the endpoints. In this paper, we prove this hypothesis in the constant positive curvature case, and show that this does not hold - in fact the effect is opposite - with constant negative curvature, see Figure~\ref{fig:intro}. Estimators based on most probable paths have been used in statistical applications in e.g.~\cite{sommer_anisotropic_2015,sommer_infinitesimal_2018-1}.

\subsection{Outline}
We start with a brief survey of anisotropic processes on manifolds, their application in geometric statistics, and the relation between geodesic distances, most probable paths, and least-squares constructions exemplified by the Fr\'echet mean. Section~\ref{sec:GeometryFrame} outlines the stochastic process and frame bundle theory used in the paper. In Section~\ref{sec:distance}, we define the sub-Riemannian structures on $\Sym^+TM$, $FM$ and $OM$ that encode infinitesimal covariance, we define path probability and quantify the effect on estimators when varying the total variance. Section~\ref{sec:dynamics} contains the dynamical equations for most probable paths and consequences of the derived system. We study most probable paths in specific examples in Sections~\ref{sec:surfaces} and~\ref{sec:examples}.
We further discuss numerical implementation of the systems and provide algorithms for estimation of mean and covariance using most probable paths in Section~\ref{sec:algorithms}.

\section{Anisotropic distributions, geometric statistics, and most probable paths} \label{sec:background}
We here give a short survey of the relation between mean estimation, infinitesimal covariance, and most probable paths in geometric statistics. Particularly, this leads to most probable paths as extremals for an objective function that generalize the Fr\'echet variance. We get estimators for diffusion means in the presence of non-trivial covariance. These estimators depend on the covariance-weighed length of the most probable paths.

\subsection{Mean values on Riemannian manifolds}
Among the most fundamental constructions in geometric statistics, the statistical analysis of manifold valued data~\cite{pennec_intrinsic_2006}, is the Fr\'echet mean \cite{frechet_les_1948}, defined as the set of minimizers of the expected square distance to a random variable $X$ on a Riemannian manifold $M$ with metric $g$ and distance $d_g$:
\begin{equation}
  E(X)
  =
  \argmin_{x\in M}\E[d_g(X,x)^2]
  \label{eq:Frechet_mean}
  \ .
\end{equation}
Unlike the Euclidean mean which can be defined in multiple equivalent ways, Riemannian manifolds have several non-equivalent notions of mean values. The diffusion mean \cite{hansen_diffusion_2021,hansen_diffusion_2021-1} is based on the characterization of the Euclidean mean value as the most likely center point of normal distributions. Distributions generated by Riemannian Brownian motions can be seen as manifold generalizations of Euclidean normal distributions. This results in the diffusion mean set
\begin{equation}
  E_t(X)
  =
  \argmin_{x\in M}\E[-\ln p_t(X;x)]
  \label{eq:diffusion_mean}
\end{equation}
where $p_t(\cdot;x)$ is the density of a Brownian motion started at $x\in M$ and evaluated at time $t>0$. Because $\lim_{t\to 0}-2t\ln p_t(y;x)=d_g(y;x)^2$, diffusion means are linked to the Fr\'echet mean in the $t\to0$ limit. For fixed positive $t$, the $2t$ factor does not affect the minima, and, under weak conditions, the sets $E_t(X)$ converge to $E(X)$ as $t\to 0$ \cite{hansen_diffusion_2021}. However, for larger $t$, the sets can deviate substantially and the Fr\'echet and diffusion means can have qualitatively different behaviours.

\subsection{Data anisotropy}
\label{sec:data_anisotropy}
The diffusion mean \eqref{eq:diffusion_mean} gives rise to the question of what happens if the manifold normal distributions that are fitted to data by maximum likelihood in \eqref{eq:diffusion_mean} have non-trivial covariance. This case is not covered by distributions generated by Brownian motions that by construction are isotropic since Brownian motions diffuse equally in all directions. To treat this question, \cite{sommer_anisotropic_2015,sommer_modelling_2017} defined anisotropic normal distributions on manifolds using Brownian-like processes, however with directionally dependent diffusion coefficients.
We write $p_t(y;x, \Sigma)$ for the density of such a process, where $x, y \in M$, $x$ represents the mean of the distribution and $\Sigma \in \Sym^+ T_x M$ is a symmetric, positive definite linear map from $T_xM$ to itself representing the covariance. We will give the construction of these densities in Section~\ref{sec:stochastics}. For a random variable $X$, the diffusion mean and covariance $(x,\Sigma)$ can now be found simultaneously by minimizing the negative log-likelihood
\begin{equation}
(  x, \Sigma) =\argmin_{(\tilde x, \tilde \Sigma) \in \Sym^+ TM}\E[-\ln p_t(X;\tilde x, \tilde \Sigma)].
  \label{eq:EtSigma}
\end{equation}
The diffusion principal component analysis (diffusion PCA) construction \cite{sommer_infinitesimal_2018-1} continues this idea by employing a maximum likelihood fit of such distributions to give a generalization of PCA to manifolds.

As for the link between the (isotropic) diffusion mean and the Fr\'echet mean, one can study the $t\to 0$ limit of \eqref{eq:EtSigma}. This turns out to have a least-squares formulation similar to \eqref{eq:Frechet_mean}, however with the squared Riemannian distance $d_g(\cdot,x)^2$ replaced by a function $d_\rho(\pi^{-1}(\cdot),(x,\Sigma))^2$ that arises from a sub-Riemannian distance on the bundle of symmetric positive definite endomorphisms of $TM$ or, alternatively, the frame bundle of $M$:
\begin{equation}
  \lim_{t\to 0}-2t\ln p_t(y;x,\Sigma)=d_\rho(\pi^{-1}(y),(x,\Sigma))^2.
\end{equation}
Here $\pi$ is the fiber bundle projection. The distance in $d_\rho$ was defined in \cite{sommer_modelling_2017}. We will revisit the definition in Sections~\ref{sec:GeometryFrame} and~\ref{sec:distance}.
The limit generalizes the standard Brownian motion small-time limit
$\lim_{t\to 0}-2t\ln p_t(y;x)=d_g(y,x)^2.$
One can therefore approximate the objective $-\ln p_t(X;x,\Sigma)$ in \eqref{eq:EtSigma} with $\frac12d_\rho(\pi^{-1}(X),(x,\Sigma))^2$. As above, for $t>0$, the factor $2t$ does not affect the minima.

Paths realizing $d_\rho(\pi^{-1}(y),(x,\Sigma))^2$ are in a certain sense most probable for the anisotropic diffusion processes 
and thus give the most probable ways of getting from the mean to observed data points. 
The cost $d_\rho(\pi^{-1}(y),(x,\Sigma))^2$ and paths realizing it thus take a dual role in both approximating the objective of \eqref{eq:EtSigma} for small $t$ and in being most-probable for the diffusion process for any $t>0$. We use both roles in the forthcoming sections.

\subsection{Sample estimators}
Let now $y_1,\dots,y_n$ be i.i.d. samples on the manifold $M$. Following \cite{sommer_modelling_2017}, consider the sample estimator 
\begin{equation}
  \argmin_{(x,\Sigma)} \frac1n \sum_{i=1}^n \left( -\ln p_t(y_i; x,\Sigma) \right),
  \label{eq:Sigma_est}
\end{equation}
of the diffusion mean \eqref{eq:EtSigma}. Again, since the density $p_t(y_i; x,\Sigma)$ generally is complex to approximate computationally, we can use the small-$t$ limit that suggests the approximation \cite{sommer_modelling_2017},
\begin{equation}
  \argmin_{(x,\Sigma)} \frac1{2n} \sum_{i=1}^n \left( d_\rho(\pi^{-1}(y_i), (x,\Sigma))^2+\ln\det_g \Sigma \right)
  \ .
  \label{eq:sample_est} 
\end{equation}
Again, most probable paths arise as the paths realizing the objective of the sample estimator.
The term $\ln\det_g \Sigma$ prevents $\Sigma$ from being arbitrarily large and thereby $d_\rho(\cdot, (x,\Sigma))$ going to zero. It corresponds to the normalization factor of the Euclidean normal distribution, and it can be interpreted as the difference between the volume forms defined by the Riemannian distance on $M$ and the sub-Riemannian distance on the bundle of symmetric positive definite endomorphisms of $TM$ determined by $\Sigma$.

\section{Frame bundle geometry and stochastic development} \label{sec:GeometryFrame}
To establish the geometric foundation for the study of most probable paths, we give a short introduction to some of the intrinsic structures that exists on the frame bundle of a Riemannian manifold and that will be used in the paper. Following this, we discuss the stochastic development procedure and its use in defining stochastic processes on manifolds.
\subsection{The frame bundle of a Riemannian manifold}
Frame bundles of Riemannian manifolds, made by enlarging the manifolds with all possible choices of bases for its tangent spaces, have two distinctive advantages. Firstly, not only does the frame bundle have a trivializable tangent bundle, it comes with a canonical choice of basis. Such a choice of basis is very useful for introducing development of stochastic processes. The second advantage is that Lie brackets of vector fields in this canonical basis can explicitly described using geometric invariants, which will be very useful for the proof of the equations for the most probable paths in Theorem~\ref{th:ProbPath}. We refer to \cite{Sha97,Hsu02} for more details.

In the discussions below, $\mathbb{R}^d$ will always be the $d$-dimensional Euclidean space and with the standard basis $e_1, \dots, e_d$. We define $\GL(d)$ as the Lie group of all invertible $d \times d$ matrices with usual matrix multiplication. Its Lie algebra $\gl(d)$ consist of all $d \times d$-matrices with the usual commutator bracket of matrices.

Let $M$ be a $d$-dimensional differentiable manifold. For any $x \in M$, consider $\GL(T_xM)$ as the space of all linear isomorphisms $u: \mathbb{R}^d \to T_xM$. Such a map can be identified with a choice of basis $u_1, \dots, u_d$ of $T_xM$ by $u_j = u(e_j)$. The frame bundle $FM = \GL(TM)$ is then the principal bundle over $M$,
$$\GL(d) \to FM \stackrel{\pi}{\to} M, \qquad FM_x = \GL(T_xM) = \pi^{-1}(x),$$
where $\GL(d)$ acts on each fiber by composition on the right. In other words, if $u \in FM_x$ corresponds to the basis $u_1, \dots, u_d$ of $T_xM$, then $u \cdot q$ corresponds to the basis $\sum_{i=1}^d u_{i} q_{i1}, \dots, \sum_{i=1}^d u_{i} q_{id}$ for any $q \in \GL(d)$.

(
Using the action of $\GL(d)$, we can associate a vector field $\xi_A$ for each $A \in \gl(d)$ by
$$\textstyle \xi_A|_u = \frac{d}{dt} u \cdot e^{tA} |_{t=0}, \qquad A \in \gl(d).$$
At each point, this is a derivative of a rotation in the fiber, and hence these get annihilated by the differential $\pi_*$ of $\pi$ as $\pi(u \cdot e^{tA}) = \pi(u)$. In fact, the vector fields $\xi_{A}$, $A\in \gl(d)$ span the vertical bundle $\calV := \ker \pi_*$ of $\pi: FM \to M$ and have bracket relations
$$[\xi_A, \xi_B] = \xi_{[A,B]}, \qquad A, B \in \gl(d).$$
We also have a tautological $\mathbb{R}^d$-valued one-form $\theta$ on $FM$ given by
$$\theta: w  \mapsto u^{-1} (\pi_* w), \qquad w  \in T_u FM.$$
In other words, $\theta$ is the result of taking a vector $w \in T_u FM$, looking at its projection $\pi_* w \in T_{\pi(u)} M$ and writing this vector in the basis $u$. Observe that the kernel of $\theta$ is exactly the vertical bundle $\calV$.

Introduce a Riemannian metric $g$ on $M$ with Levi-Civita connection~$\nabla$.
For every smooth curve $t \mapsto \gamma(t)$, there is a $\nabla$-parallel frame $u(t) \in FM_{\gamma(t)}$ uniquely determined by its value at an initial point. Let $\mathcal{H}$ be the set of derivatives of such curves. Then
\begin{equation} \label{H+V} TFM = \calH \oplus \calV,\end{equation}
since the derivative of any parallel frame is uniquely determined by the derivative of the underlying curve.
The subbundle $\calH$ is invariant under the group action, and we can hence define a corresponding principal connection $\omega:TFM\to \gl(d)$ given by
$$\calH = \ker \omega, \qquad \omega(\xi_A) = A.$$
Invariance under the group action means that $\calH_{u \cdot q} = \calH_u \cdot q$ for any $u \in FM$, $q \in GL(d)$, which in term can be expressed by the principal connection $\omega$ as the identity $\omega(w \cdot q) = \Ad(q^{-1}) \omega(w) = a^{-1} \omega(w) q$.

For any vector $v \in T_xM$ and $u \in FM_x$, define $h_uv \in \calH_u$ as the unique vector projecting to $v$. Similarly, for any vector field $V \in \Gamma(TM)$, define a vector field $hV \in \Gamma(TFM)$ by $hV|_u = h_u V|_{\pi(u)}$, which is called the horizontal lift. Finally, for an element $a \in \mathbb{R}^d$, we define the canonical horizontal vector field $H_a$ as the unique vector field satisfying
\begin{equation} \label{Xvector} \theta(H_a) = a, \qquad \omega(H_a) = 0.\end{equation}
If $a = \sum_{i=1}^d a_i e_i$, then $H_a$ is related to the previous mentioned horizontal lifts by
$$H_a|_u = \sum_{i=1}^d a_i h_u u_i = \sum_{i=1}^d a_i H_i,$$
where $H_i = H_{e_i}$.

The above definitions have the following local representation. Choose a local orthonormal basis $V_1, \dots, V_d$ of $TM$ and define a corresponding local trivialization
$$FM \cong M \times GL(d), \qquad u \in FM_x \mapsto (x, (u_{ij})), \qquad u_j = \sum_{i=1}^d u_{ij} V_i.$$
If we write $\nabla_{V_i} V_j = \sum_{k=1}^d \Gamma_{ij}^k V_k$ for the Christoffel symbols, then
\begin{align*}
\xi_A & \textstyle =  \sum_{i,j,r=1}^d u_{ir} A_{rj} \frac{\partial}{\partial u_{ij}}, \\
hV_k & \textstyle = V_k - \sum_{i,j,r=1}^d u_{rj} \Gamma_{kr}^i \frac{\partial}{\partial u_{ij}}, \\
H_a & \textstyle = \sum_{i,j=1}^d a_j u_{ij} hV_i .
\end{align*}
Write $R$ for the curvature of the Levi-Civita connection. Using the above formulas, we have the local identities
\begin{align*}
{[hV_k, hV_l]} & \textstyle = h[V_k, V_l] - \sum_{i,j, r=1}^d u_{ir} \langle u^{-1} R(V_k, V_l) u_j, e_r \rangle \frac{\partial}{\partial u_{ij}} ,\\
{[\xi_A, hV_l]} & =   0.
\end{align*}
Define $\underline{R}: FM \to (\mathbb{R}^{d,*})^{\otimes 3} \otimes \mathbb{R}^d$ as the scalarization of the curvature $R$, given by
$$\underline{R}(u)(a,b)c = u^{-1} R(u(a), u(b) ) u(c), \qquad u \in FM, a,b,c \in \mathbb{R}^d.$$
We use the previous local identities for the Lie brackets to give global formulas for our canonical basis of vector fields on $FM$
\begin{align} \label{Hxibrackets}
{[H_a, H_b]} & =- \xi_{\underline{R}(a,b)} & {[\xi_A, H_a]} & = H_{Aa} , \qquad a,b \in \mathbb{R}^d, A \in \gl(d).
\end{align}
The corresponding identities for forms are given as
\begin{align} \label{CartanEq}
d\omega + \frac{1}{2} [\omega, \omega] &= \Omega, & d\theta + [\omega , \theta] & = 0, 
\end{align}
where the curvature form $\Omega$ is a two-form with values in $gl(d)$ which vanishes on $\calV$ and satisfies
$$\Omega(H_a, H_b)|_u = \underline{R}(u)(a,b).$$

If we restrict ourselves to the orthonormal frame bundle $OM \subseteq FM$, then orthonormal frames remain orthonormal under parallel transport with respect to the Levi-Civita connection. Hence, the above formalism also makes sense if we only consider orthonormal frames. For this reason, by slight abuse of notation, we will use the symbols $\calH$, $\theta$ and $\omega$ for the restrictions of these to~$OM$.

\subsection{Stochastic processes and development} \label{sec:stochastics}
Let $t \mapsto B_t = (B_t^1, \dots, B_t^d)$ be the standard Brownian motion on $\mathbb{R}^d$, meaning in particular that $B_t$ is normally distributed as $\mathcal{N}(0, t1_d)$ for a fixed $t$.
Throughout this section, we will assume that $M$ is a compact manifold which by \cite{Shi82} will be sufficient for the solutions of the SDEs below on $FM$ to have infinite lifetime. See Remark~\ref{re:Lifetime} for the noncompact case.

Recall the definition of the frame bundle $FM$ of a $d$-dimensional Riemannian manifold $(M,g)$. For a given initial frame $u \in FM$, we define the process $t \mapsto U_t =  U_t(u)$ as the solution of the Stratonovich SDE,
\begin{equation}
dU_t  =  H_{\circ dB_t}|_{U_t} = \sum_{i=1}^d H_i|_{U_t} \circ dB_t^i, \qquad U_0 = u.
  \label{eq:development}
\end{equation}
Define $X_t(u) = \pi(U_t(u))$ as its projection to $M$. Here $X_{t}$ is the \emph{development} of~$B_t$, and the development can be reversed in the sense that $B_t$ can be found from~$X_t$ and the initial condition $u$. In this case, $B_t$ is denoted \emph{the anti-development} of~$X_t$. The stochastic development \eqref{eq:development} has a deterministic counterpart in the ODE
\begin{equation}
\dot u(t)  =  H_{v(t)}|_{u(t)} = \sum_{i=1}^d H_i|_{u(t)} \dot b^i(t), \qquad u(0) = u.
  \label{eq:development_det}
\end{equation}
for an absolutely continuous path $b(t)$ in $\mathbb{R}^n$ and with $u(t)$ being a parallel frame along a path $x(t)$. If $\ptr_t: T_x M \to T_{x(t)} M$ denotes parallel transport along $x(t)$, so that we may write $u(t) = \ptr_t u$, then this equation can be written as $\dot x(t) = \ptr_t u \dot b(t)$ with $u(t)$ then being determined by parallel transport.

Let $\Sym^+  TM$ denote positive definite symmetric endomorphisms of $TM$.
Define a map $\Sigma: FM \to \Sym^+ TM$, $u \mapsto \Sigma_u$ by
\begin{align*}
\langle u^{-1} w_1, u^{-1} w_2 \rangle & = \langle \Sigma_u^{-1} w_1, w_2 \rangle_g, \qquad u \in FM_x, w_1, w_2 \in T_x M.
\end{align*}
We observe that $\Sigma_u$ is always invertible and symmetric, and equals the identity on~$T_x M$ if $u$ is orthonormal. Furthermore, if $q \in \Ort(d)$, then $\Sigma_{u \cdot q} = \Sigma_u$. Similarly, since the Brownian motion is rotationally invariant, we can identify $X_t(u \cdot q)$ with $X_t(u)$ for $q \in O(d)$ and write $X_t(u) = X_t(\Sigma_u)$.

Next, let $\Sym^+ \mathbb{R}^d$ consist of symmetric, positive definite $d \times d$-matrices. We define a map $S: FM \to \Sym^+ \mathbb{R}^d$, $u \mapsto S_u$ by
$$\langle S_u^2 a, b \rangle = \langle u(a), u(b) \rangle_g, \qquad a, b \in \mathbb{R}^d.$$
We observe that $u^{-1} \Sigma_u u= S_u^2$. Furthermore, since $\langle u(t) a, u(t) b \rangle_g$ is constant in any parallel frame
\begin{equation} \label{HaS0} H_a S = 0, \qquad \text{for any $a\in \mathbb{R}^d$}.\end{equation}
Hence, for some fixed $S \in \Sym^+ \mathbb{R}^d$, if we define
$$F^SM = \{ \tilde u \in FM \, : \, S_{\tilde u} = S \},$$
then $U_t(u)$ takes values in $F^{S_u}M$.

There is a diffeomorphism $i_S$ from $F^S M$ to the orthonormal frame bundle $OM \subseteq FM$ given by
\begin{equation} \label{iS} i_S(u) = u \cdot S^{-1}. \end{equation}
Let us solve the following SDE on the orthonormal frame bundle. For any $S \in \Sym^+ \mathbb{R}^d$ and $f \in OM$, define $\hat U_t = \hat U_t(S,f)$ as the solution of
\begin{equation}
\label{tildeX}
d\hat U_t = H_{SdB_t} |_{\hat U_t} = \sum_{i,j=1}^d S_{ij} H_i |_{\hat U_t} \circ dB^j_t, \qquad \hat U_0 = f.
\end{equation}
Furthermore, we write $\hat X_t(S,f) = \pi(\hat U_t(S,f))$. Observe that for fixed $t$, $SB_t$ is distributed as $\mathcal{N}(0, tS^2)$. We also give the following observations.
\begin{proposition} \label{prop:YX}
\begin{enumerate}[\rm (a)]
\item For any $u \in FM$, the processes $X_t(u)$ and $\hat X_t(S_u, u \cdot S_u^{-1})$ are indistinguishable.
\item For any $\Sigma \in \Sym^+ T_x M$, $x \in M$, let $f \in OM_x$ be any orthonormal frame, and write $S^2 = f^{-1} \Sigma f$. Then $X_t(\Sigma)$ and $\hat X_t(S , f)$ are indistinguishable.
\end{enumerate}
\end{proposition}

\begin{proof}
The result in (a) follows by observing that $U_t(u) \cdot S_u^{-1}$ solves the equation in \eqref{tildeX}. For (b), let $f \in OM$ be chosen. Let $\tilde u$ be any frame with $\Sigma_{\tilde u} = \Sigma$ and define $\tilde f = \tilde u \cdot S_{\tilde u}^{-1} = f \cdot q$ for some $q \in \Ort(d)$. If $u = \tilde u \cdot q^{-1}$, then
$$\Sigma_{\tilde u} = \Sigma_{u}, \qquad S_{u} = q S_{\tilde u} q^{-1},$$
and so $u \cdot S_u^{-1} = \tilde u \cdot S_{\tilde u}^{-1} \cdot q^{-1} = f$. We furthermore have that $S_u^2 = u^{-1} \Sigma_u u = S_u^{-1} f^{-1} \Sigma_u f S_u$, so $S_u^2 = f^{-1} \Sigma_u f$. Since $X_t(u)$ and $X_t(u \cdot q)$ are indistinguishable, the result follows.
\end{proof}

\begin{definition}[\cite{sommer_anisotropic_2015}]
Let $(M,g)$ be a Riemannian manifold, with $x \in M$ and $\Sigma \in \Sym^+ T_x M$ arbitrary. We consider the normal distribution $\mathcal{N}(x, \Sigma)$ on $M$ as the distribution of $X_1(\Sigma)$. 
\end{definition}

\begin{remark}
\begin{enumerate}[\rm (a)]
\item As a consequence of Proposition~\ref{prop:YX}, no generality is lost by the choice of $t=1$ in the definition of $\mathcal{N}(x, \Sigma)$ since $X_t(\Sigma)=X_1(t\Sigma)$.
\item If $f$ is chosen to be an eigenframe of $\Sigma$ in Proposition~\ref{prop:YX}~(b), then $S = \Lambda = \diag\{ \lambda_1, \dots, \lambda_d\}$ is a diagonal matrix.
\end{enumerate}
\end{remark}

\begin{remark} \label{re:Lifetime}
If $M$ is non-compact, the processes $X_t(\Sigma)$ and $\hat X_t(S,f)$ may only be defined up to an exploding time, but this does not change anything about our above conclusions.
\end{remark}

\begin{remark}[Summary and comparison with the Euclidean case] \label{re:RnComp}
To compare with the Euclidean setting and summarize the section: We are interested in having an intrinsic way of computing the mean and variance on a curved space where vector space structure is not available. The approach is to consider the normal distribution $\calN(0, \Sigma)$ in $\mathbb{R}^d$ as the density of the stochastic process $\Sigma^{1/2} B_t$ at time $t=1$ where $B_t$ is a standard Brownian motion. If $u_1, \dots, u_d$ is a choice of orthonormal basis of $T_x M$ where $x$ is a point on the manifold $M$, we can use this basis to consider $\Sigma^{1/2}B_t$ as a process in $T_xM$ with $\Sigma$ now an endomorphism of $T_xM$. We can then use the construction of the orthonormal frame bundle to ``steal'' the property of having a canonical basis. This allow us to define a process $X_t$ on the manifold by finding the process whose differential equals that of $\Sigma^{1/2} B_t$ in the basis $H_1, \dots, H_d$ and projecting it back to the manifold. Stopping this again at time $t=1$ gives us an intrinsic way of transferring the normal distribution $\calN(0, \Sigma)$ to a curved space.

Alternatively, if $\sigma_1, \dots, \sigma_d$ denotes the columns of $\Sigma^{1/2}$, we can write $\Sigma^{1/2} B_t$ as a standard Brownian motion in this non-orthonormal basis. We can now copy the process above, but now using the general frame bundle $FM$ to equivalently obtain the density $\calN(x,\Sigma)$ on the manifold.

The construction of the map $u \mapsto S_u$ is something that is necessary only in the non-flat case, as the result of parallel transport of $\Sigma_u$ from $x$ to a different point $y$ will depend on path in general, while $S_{u(t)}$ will remain constant along any parallel path. Also, for the case of $\mathbb{R}^d$, the curvature $R$ vanishes, meaning that $\calH$ is Frobenius integrable, i.e. $[\calH, \calH] \subseteq \calH$. This means that there is a foliation $\mathcal{F}$ of $FM$ where each leaf is tangent to $\calH$, and since $TFM = \calH \oplus \ker \pi_*$, each leaf is diffeomorphic to $M$.
\end{remark}

\section{Sub-Riemannian distance and most probable paths} \label{sec:distance}
\subsection{Sub-Riemannian structure on symmetric endomorphisms}
We now define the sub-Riemannian distance $d_\rho$ that was used in Section~\ref{sec:data_anisotropy} and that enters in the Onsager-Machlup functional for the most probable paths. By a sub-Riemannian manifold, we mean a manifold $N$ with a smoothly varying inner product $\rho = \langle \cdot , \cdot \rangle_{\rho}$ defined only on a subbundle $E\subseteq TN$. The subbundle $E$ can be thought of as the ``permissible directions'' on the manifold as only curves that are tangent to $E$ have a well-defined length. The distance $d_\rho(x,y)$ between two points are then found by taking the infimum over all the lengths of the curves connecting $x$ and $y$ that are also tangent to $E$. Further details of sub-Riemannian structures are outlined in Appendix~\ref{sec:SR}. In our discussion below, the manifold $N$ will be either $FM$, $OM$ or $\Sym^+ TM$ with the permissible directions being those that are the result of parallel transport with respect to the Levi-Civita connection.

We consider the fiber bundle $\pi: \Sym^+ TM \to M$ as defined in Section~\ref{sec:stochastics}. Just as in \eqref{H+V}, we have a decomposition $T\Sym^+ TM = E \oplus \ker \pi_*$, where $E$ is the derivatives of all curves $\Sigma(t)$ that are parallel along their projection $\pi(\Sigma(t)) = \gamma(t)$. We then define a sub-Riemannian metric $\rho$ on $E$ by
$$\left\| \dot \Sigma(t) \right\|_\rho^2 = \left\langle \dot \Sigma(t), \dot \Sigma(t) \right\rangle_\rho = \left\langle \Sigma(t)^{-1} \dot \gamma(t), \dot \gamma(t) \right\rangle_g, $$
where $\Sigma(t)$ is any curve tangent to $E$. Denote the corresponding sub-Riemannian distance by $d_\rho$.

Let $\Sigma \in \Sym^+ T_x M$ be a fixed element for $x \in M$. For any curve $\gamma:[0,T] \to M$ with $\gamma(0) = x$, introduce notation $\ptr_t : T_x M \to T_{\gamma(t)} M$ for the parallel transport along the curve. We define $L^{\Sigma}(\gamma) = L^\rho(\Sigma(\cdot ))$ as the length of the curve
\begin{equation} \label{Sigmaptr} \Sigma(t) =\Sigma_{\ptr_t} = \ptr_t \Sigma \ptr_t^{-1}, \end{equation}
with respect to $\rho$. Then for any $y \in M$,
$$d_\rho(\pi^{-1}(y), (x, \Sigma)) = \inf_{\begin{subarray}{c} \gamma(0) = x, \\ \gamma(T) = y \end{subarray}} L^\Sigma(\gamma).$$
This equation defines the map $d_\rho$ as used in Section~\ref{sec:background}. Note that the choice of $T$ in the interval $[0,T]$ does not affect the distance, as we can reparametrize a curve to be defined on any given interval.

\subsection{Alternative description} \label{sec:Alternative}
The sub-Riemannian length $L^\rho(t)(\Sigma(\cdot))$ can also be realized in the following two alternative ways. Let $\Sigma(t)$ be a curve in $\Sym^+ TM$ that is parallel along its projection $\gamma(t)$ with $\Sigma(0) = \Sigma^x \in \Sym^+ T_x M$. Let $u^x \in FM$ be any frame with $\Sigma_{u^x} = \Sigma^x$ and define $u(t)$ by parallel transport along $\gamma(t)$. It then follows that $\Sigma_{u(t)} = \Sigma(t)$ for all $t$. It also follows that $u(t)$ is tangent to the bundle $\calH$ in Section~\ref{sec:GeometryFrame}. We can then define a sub-Riemannian metric $\tilde \rho$ on $\calH$ by
\begin{align*}
& \| \dot \Sigma(t) \|_{\rho}^2 = \langle \Sigma(t)^{-1} \dot \gamma(t), \dot \gamma(t) \rangle_g \\
& = \langle u(t)^{-1} \dot \gamma(t), u(t)^{-1} \dot \gamma(t) \rangle = \langle \theta(\dot u(t)) , \theta(\dot u(t)) \rangle =: \| \dot u(t)\|^2_{\tilde \rho} ,
\end{align*}
with the property $L^\rho(\Sigma(\cdot)) = L^{\tilde \rho}(u(\cdot))$. Observe that $(FM, \calH, \tilde \rho)$ has a global orthonormal basis $H_1, \dots, H_d$, in contrast to $\Sym^+TM$.

We can also consider the problem on the orthonormal frame bundle $OM$. Let $f^x \in OM_x$ be any initial frame and define $S^2 = (f^x)^{-1} \Sigma^x f^x$. If we define $f(t)$ by parallel transport, then $S^2 = f(t)^{-1} \Sigma(t) f(t)$ for any $t$. We introduce a corresponding sub-Riemannian metric $\rho_S$ on $\calH$, now considered as the bundle of derivatives of parallel orthogonal frames. We define it by
\begin{align*}
& \| \dot \Sigma(t) \|_{\rho}^2 = \langle S^{-2} f(t)^{-1} \dot \gamma(t), f^{-1} \dot \gamma(t) \rangle_g =: \| \dot f(t)\|^2_{\rho_S}.
\end{align*}
In other words, $(OM, \calH, \rho_S)$ has global orthonormal basis $H_{Se_1}, \dots, H_{Se_d}$. By the above equality, the lengths $L^\rho(\Sigma(\cdot))$ and $L^{\rho_S}(f(\cdot))$ coincide.

In what follows, we will often state our results using the formulation on $\Sym^+TM$, as this does not require any choice of initial frame. However, we will usually present our proofs on $OM$, as this reduces the problem to a space of minimal dimension and provides access to a global basis for the horizontal subbundle.

\subsection{Path probability and the Onsager-Machlup functional}
Before investigating most probable paths on $FM$ or $\Sym^+ TM$, we review the construction of the Onsager-Machlup functional and its relation to path probability and most probable path.
Let $B_t$ be a Euclidean Brownian motion. The Onsager-Machlup functional measures the probability that realizations of $B_t$ sojourns around smooth paths in the sense of staying in $\ve>0$ diameter cylinders. More precisely, for paths $b\in C^1([0,T],\mathbb R^d)$, define
\begin{equation}
  \mu_\ve(b)
  =
  P(\|B_t-b(t)\|<\ve\ \forall t\in[0,T])
  .
\end{equation}
It can now be shown \cite{fujita_onsager-machlup_1982} that $\log \mu_\ve(b)$ tends to $c_1+c_2/\ve^2+\int_0^T L(b(t),\dot{b}(t))dt$ for constants $c_1,c_2$ as $\ve\to 0$ and with $L(b(t),\dot{b}(t))=-\frac12\|\dot{b}(t)\|^2$. The function $L$ is denoted the Onsager-Machlup functional. One here recognizes the usual Euclidean energy of~$b$ in the integral over $L$. Paths between two points maximizing $\int_0^T L(b(t),\dot{b}(t)) \, dt$ are termed most probable. 

If instead $b\in C^1([0,T],M)$ is a curve on $M$, the Onsager-Machlup functional changes to $L(b(t),\dot{b}(t))=-\frac12\|\dot{b}(t)\|_g^2+\frac1{12}S(b(t))$ where $S$ is the scalar curvature of $M$. Most probable paths on manifolds thus minimize a functional that in addition to the path energy includes the integral of the scalar curvature along the path. In case the scalar curvature is constant over $M$, most probable paths and geodesics thus coincide.

\subsection{Path probability with development}
We will now apply the Onsager-Machlup theory, however for paths on $FM$ and $\Sym^+ TM$.
Consider the stochastic process $X_t(\Sigma)$ defined in Section~\ref{sec:stochastics}. The process is generated by a Euclidean Brownian motion $B_t$ through the SDE in \eqref{eq:development}.
We look at paths $b(t)$ such that the corresponding development \eqref{eq:development_det} of $b$ starts at $u\in FM$ with $\Sigma_u=\Sigma$ and ends with the projection $\gamma(t)=\pi(u(t))$ to $M$ satisfying $\gamma(T)=y$. We then apply the Onsager-Machlup functional on the anti-development $b(t)$ which is a path in Euclidean space and hence $L(b(t),\dot{b}(t))=-\frac12\|\dot{b}(t)\|^2$. Then $\gamma(t)$ is termed most probable for $Y_t(\Sigma)$ given by \eqref{eq:development} if it realizes
\begin{equation}
 \max_{\begin{subarray}{c} \Sigma_{u(0)}=\Sigma \\ \gamma(T)=y
 \end{subarray}}
  -\log \mu_\ve(b)
  =
  \min_{\begin{subarray}{c} \Sigma_{u(0)}=\Sigma, \\ \gamma(T)=y \end{subarray}}
  \int_0^T\frac12\|\dot{b}(t)\|^2 dt
  =
  d_\rho(\pi^{-1}(y),(y,\Sigma))^2
\end{equation}
for $y\in M$.
Extremal paths for $L^\rho(\Sigma(\cdot))$ thus have a probabilistic characterization as being \emph{most probable} with respect to $X_t(\Sigma)$.

\begin{remark}[Euclidean comparison]
Notice that we are here looking for the most probable path $b(t)$ in $\mathbb{R}^d$ given that the endpoint of the developed curve $\gamma(t)$ is $y$. If $M = \mathbb{R}^d$, then $\gamma(t) = x + b(t)$ and minimizers of the above problem are always geodesics, no matter the choice of $\Sigma$. This illustrate the fact that on $\mathbb{R}^d$, $d_\rho((x, \Sigma), (y, \tilde \Sigma))$ is only finite when $\Sigma$ equals $\tilde \Sigma$ when written in the usual coordinates and that the distance then is $d_\rho((x,\Sigma), (y, \Sigma)) = \| \Sigma^{-1/2}(x-y)\|$. This will not be the case for a general Riemannian manifold $(M,g)$.
\end{remark}

\begin{remark}[Comparison to the manifold Onsager-Machlup functional]
    When $\Sigma$ is orthonormal, $L(b(t),\dot{b}(t))=-\|\dot{\gamma}(t)\|^2$. We thus see that the application of the Onsager-Machlup functional on the anti-development $b(t)$ of $\gamma(t)$ deviates from the standard manifold construction that includes the scalar curvature term $\frac1{12}S(\gamma(t))$. 
\end{remark}

\subsection{Covariance scaling and optimal estimators}
Consider now a set of samples $y_1, \dots, y_n$, and let us return to finding the small-t limit of the most fitting density as described in \eqref{eq:sample_est}. If we write $\Sigma = C \Sigma'$ with $\det \Sigma' =1$, then writing $d_{\rho}(\pi^{-1}(y_i) , (x,\Sigma)) = d_{\rho}(\pi^{-1}(y_i) , \Sigma)$ for the sake of more compact formulas, then
\begin{align*}
& \frac1{2n} \sum_{i=1}^n \left( d_{\rho}(\pi^{-1}(y_i) , \Sigma)^2+\ln\det_g \Sigma \right) = \frac{d}{2} \ln C + \frac{1}{2nC} \sum_{i=1}^n  d_{\rho}(\pi^{-1}(y_i) , \Sigma')^2 
\end{align*} 
Here we have used that if $\Sigma(t) \in \Sym^+ T_{\gamma(t)} M$ is parallel along its projection $\gamma(t)$, then $C \Sigma(t)$ is still parallel along $\gamma(t)$. Hence, for any positive constant $C >0$, so we have $L^\rho(C \Sigma(\cdot)) = C^{-1/2} L^\rho(\Sigma(t))$. Finally $\pi^{-1}(y)$ is invariant under multiplication of positive scalars, meaning that $d_{\rho}(\pi^{-1}(y_i) , C\Sigma') = C^{-1/2} d_{\rho}(\pi^{-1}(y_i) , \Sigma')$.

Using this expressing, we see that the optimal choice for $C$ is
\begin{equation} \label{C} C = \frac{1}{nd} \sum_{i=1}^n d(\pi^{-1}(y_i), \Sigma')^2. \end{equation}
Furthermore,
\begin{align*}
\Sigma' & = \argmin_{\begin{subarray}{c} \tilde \Sigma \in \Sym^+ TM \\ \det \tilde \Sigma = 1  \end{subarray}} \left. \left(  \ln C + \frac{1}{nC} \sum_{i=1}^n  d_{\rho}(\pi^{-1}(y_i) , \tilde \Sigma)^2 \right) \right|_{C= \frac{1}{n} \sum_{i=1}^n d(\pi^{-1}(y_i), \tilde \Sigma)^2 }\\
& = \argmin_{\begin{subarray}{c} \tilde \Sigma \in \Sym^+ TM \\ \det \tilde \Sigma = 1  \end{subarray}}   \left(  \ln \left( \sum_{i=1}^n  d_{\rho}(\pi^{-1}(y_i) , \tilde \Sigma)^2 \right)  -  \ln n + 1  \right) ,
\end{align*} 
which implies that
\begin{equation} \label{SigmaPrime}
\Sigma' =  \argmin_{\begin{subarray}{c} \tilde \Sigma \in \Sym^+ TM \\ \det \tilde \Sigma = 1  \end{subarray}}  \sum_{i=1}^n  d_{\rho}(\pi^{-1}(y_i) , \tilde \Sigma)^2 .\end{equation}

The result is the following reduced optimization problem for sample estimators. The separate optimization for $\Sigma'$ and total variance $C$ can ease numerical optimization, see Section~\ref{sec:algorithms}.
\begin{proposition}
  The sample estimate \eqref{eq:sample_est} can be found by solving \eqref{SigmaPrime} to obtain optimal $\Sigma'\in \Sym^+ TM$, $\det \tilde \Sigma = 1$ followed by setting $\Sigma=C\Sigma'$ with $C$ given by~\eqref{C}.
\end{proposition}

\section{Dynamics} \label{sec:dynamics}
\subsection{Equations for most probable paths}
We can now state and prove the main theorem of the paper that characterizes the dynamics of most probable paths. In the follow subsections, we describe the most important direct consequences of the theorem.
Recall that the covariant derivative along the curve $\gamma$ is given by $D_t = \nabla_{\dot \gamma(t)}  = \ptr_t \frac{d}{dt} \ptr_t^{-1} $. 

\begin{theorem} \label{th:ProbPath}
Assume that $\gamma:[0,T] \to M$ is the most probable curve from $x$ to $y = \gamma(T)$ with respect to the covariance $\Sigma \in \Sym^+ T_x M$. If $\Sigma_{\ptr_t}$ is as in \eqref{Sigmaptr}, then $\gamma$ (or a reparametrization of $\gamma$) solves the equation
\begin{equation} \left\{ \qquad
  \begin{split}
D_t \dot \gamma(t) & =  \Sigma_{\ptr_t}  R\left(  \chi(t) \right) \dot \gamma(t), \\
D_t \chi(t) & = \dot \gamma(t) \wedge \Sigma_{\ptr_t}^{-1} \dot \gamma(t), \, \,
\chi(T) = 0 \in \wedge^2 T_xM
  \end{split} \right.
  \label{eq:mpp}
\end{equation}
\end{theorem}
Note the explicit role of the Riemannian curvature tensor $R$ in the dynamics for~$\gamma$. The infinitesimal covariance $\Sigma$ affects both the $\gamma$ and $\chi$ evolutions.

\begin{proof}
Recall the notation on the frame bundle introduced in Section~\ref{sec:GeometryFrame}, and in particular the equations \eqref{CartanEq}. Let $\Sigma \in \Sym^+ T_x M$ and $y \in M$ be fixed. Using the realization in Section~\ref{sec:Alternative}, we consider the problem on $OM$. Let $f^x \in OM_x$ be any orthonormal frame and define $S^2 = (f^x)^{-1} \Sigma f^x$. We then want to find a curve $f(t)$ defined on $[0,T]$ such that $f(t)$ is parallel along its projection $\gamma(t)$ in $M$, $f(0) = f^x$ and that $f(t)$ is an extremal with respect to the length
$$L^\Sigma(\gamma) = \int_0^T \| S^{-1} f(t)^{-1} \dot \gamma(t) \|\, dt,$$
among all such curves with $\gamma(T) = y$.

 Consider the Hilbert manifold $\mathcal{C}(f^x)$ of absolutely continuous $\calH$-horizontal curves $f:[0,T] \to OM$ with $f(0)= f^x$  and with $\dot f(t)$ in $L^2$, see Appendix~\ref{sec:SR} for details. If $s \mapsto f_s(\cdot)$ is a smooth curve in $\mathcal{C}(f^x)$ then since
$$f_s(0) = f^x, \qquad \text{ and } \qquad \omega(\dot f_s(t))= 0,$$
we know that $\partial_s f_s(0) =0$ and,
$$\partial_t \omega(\partial_s f_s) = d\omega(\dot f_s, \partial_s f_s) = \Omega\left(\dot f_s, \partial_s f_s\right) = \underline{R}(f_s)(\theta(\dot f_s), \theta(\partial_s f_s)).$$
Any tangent vector of $\mathcal{C}(f^x)$ at $t \mapsto f(t)$ can be considered as a vector field $W(t)$ along $f(t)$, such that
\begin{equation} \label{Vvf} W(t) = \left. \left(H_{w(t)} + \xi_{\int_0^t \underline{R}(f(\tau))(\theta(\dot f(\tau)), w(\tau)) \, d\tau} \right) \right|_{f(t)} , \qquad w(0) = 0,\end{equation}
for some $w \in L^2([0,T],\mathbb{R}^d)$ with $w(0) = 0$.

Consider now the endpoint map $\Pi: \mathcal{C}(f^x) \to M$ given by $f(\cdot) \mapsto \pi(f(T))$. We then see that the differential at $f(\cdot )$ for a vector field $W$ as in~\eqref{Vvf} is
$$\Pi_* W = \sum_{i=1}^d w_i(T) f_i(T).$$
It follows that every point $y \in M$ is a regular value of $\Pi$, and so the preimage $\mathcal{C}(f^x, \pi^{-1}(y)) = \Pi^{-1}(y)$ of curves from $f^x$ to $\pi^{-1}(y)$ is a Hilbert manifold from the inverse function theorem. Tangent vectors are then vector fields $W$ as in~\eqref{Vvf} with the extra restriction that $w(T) =0$.

To complete our computation for the first order condition for optimality, we will introduce some notation. We identify elements in $\wedge^2 \mathbb{R}^d$ with elements in $\so(d)$ such that for $a, b, c \in \mathbb{R}^d$, the two-vector $a \wedge b$ is identified with the map in $\so(d)$ given by
$$(a \wedge b)c = \langle a, c\rangle b - \langle b, c \rangle a.$$
Conversely, a matrix $A = (A_{ij})$ is identified with the two-vector
$$\sum_{i<j} A_{ij} e_j \wedge e_i = \frac{1}{2} \sum_{i,j=1}^d A_{ij} e_j \wedge e_i.$$
Define an inner product on $\so(d)$ by
\begin{equation} \label{soIP} \langle A, B \rangle = - \frac{1}{2}\tr AB, \qquad A, B \in \so(d).\end{equation}
We note the relation
$$\langle A, a \wedge b \rangle =\langle Aa, b\rangle.$$
Furthermore, if $f$ is any orthonormal frame, then we will have
$$\langle \underline{R}(f)(a,b) , A \rangle = \langle \underline{R}(f)(A), a \wedge b \rangle.$$

We look for critical points of the energy functional
$$E(f(\cdot)) = \frac{1}{2}  \int_0^T \left|S^{-1}  \theta(\dot f(t))\right|^2 \, dt,$$
on the manifold $\mathcal{C}(f^\mu, \pi^{-1}(y))$.
Let $s \mapsto f_s$ be any smooth curve in $\mathcal{C}(f^\mu, \pi^{-1}(y))$ with $f_0(t) = f(t)$ and $\partial_s f |_{s=0} (t)= W(t)$. Write
$$\theta(W(t)) = w(t), \qquad \theta(\dot f(t)) = v(t),$$
and recall that then
$$w(0) = w(T) = 0, \qquad \omega(W(t)) = \int_0^t \underline{R}(f(\tau))(v(\tau), w(\tau)) d\tau.$$
If $f(t)$ is a local minimum of the energy, then
\begin{align*}
& 0 = \partial_s E(f_s)|_{s=0} =   \int_0^T \langle S^{-2} \theta(\dot f_s(t)) ,  \partial_s \theta(\dot f_s(t)) \rangle \, dt|_{s=0} \\
& =  \int_0^T \left\langle S^{-2}  \theta(\dot f_s(t)) ,   \partial_t \theta(\partial_s f_s(t)) \right\rangle \, dt|_{s=0} +  \int_0^T \langle S^{-2}  \theta(\dot f_s(t)) ,  d\theta(\partial_s f_s, \dot f_s(t)) \, dt|_{s=0} \\
& =  \int_0^T \left\langle S^{-2}  v(t) ,   \dot w(t) \right\rangle \, dt  -  \int_0^T \langle S^{-2}  v(t) ,  \omega(W(t)) v(t) \rangle \, dt \\
& =  - \int_0^T \left\langle S^{-2}  \dot v(t) , w(t) \right\rangle \, dt  - \int_0^T \left\langle v(t) \wedge S^{-2}  v(t) ,  \int_0^t \underline{R}(f(\tau))(v(\tau), w(\tau)) d\tau  \right\rangle \, dt.
\end{align*}
Defining $A(t) = \int_0^t v(\tau) \wedge S^{-2} v(\tau) \, d\tau$ as a curve in $\so(d)$, and using integration by parts
\begin{align*}
0 & =  - \int_0^T \left\langle S^{-2}  \dot v(t) ,   w(t) \right\rangle \, dt  -   \left\langle A(T) ,  \int_0^T \underline{R}(f(t))(v(t), w(t)) dt \right\rangle \, dt \\
&  \qquad + \int_0^T \left\langle A(t) ,  \underline{R}(f(t))(v(t), w(t)) \right\rangle dt \\
& =  \int_0^T \left\langle - S^{-2}  \dot v(t) + \underline{R}(f(t))(A(t)- A(T)) v(t) ,  w(t) \right\rangle\, dt .
\end{align*}
It follows that
\begin{align*}
\dot v(t) = S^{2} \underline{R}(f(t)) (A(t) - A(T)) v(t) .
\end{align*}
The result follows by defining $\chi(t) = \frac{1}{2}\sum_{i,j=1}^d (A_{ij}(t) - A_{ij}(T)) f_j(t) \wedge f_i(t)$.
\end{proof}

\begin{remark} \label{re:Parametrization}
The solutions in Theorem~\ref{th:ProbPath} are the ones parametrized such that $| \Sigma_{\ptr_t}^{-1/2} \dot \gamma(t) | =c $ is a first integral. We can see this directly
 from
$$
\frac{1}{2} \frac{d}{dt} | \Sigma_{\ptr_t}^{-1/2} \dot \gamma(t) |^2 = \langle \Sigma_{\ptr_t}^{-1} D_t \dot \gamma(t) , \dot \gamma(r) \rangle  = -  \left\langle R\left( \chi(t) \right) \dot \gamma(t), \dot \gamma(t) \right\rangle =0. $$
It follows that $L^\Sigma(\gamma) = cT$.
 \end{remark}

\subsection{Consequences of Theorem~\ref{th:ProbPath}}
We look at some immediate consequences of our previous result. We will first make a statement about normal geodesics, see Appendix~\ref{sec:SR} for definition.

  \begin{corollary}
    Consider $\Sym^+ TM$ with the sub-Riemannian metric $\rho$. Solutions of \eqref{eq:mpp} are normal geodesics of $\rho$. Conversely, normal geodesics of $\rho$ are the solutions of \eqref{eq:mpp} with the condition $\chi(T) = 0$ omitted.
\end{corollary}
\begin{proof}
This results follow from a similar computation as in the proof of Theorem~\ref{th:ProbPath}, but from a Hamiltonian rather than a Lagrangian perspective. We leave the details in Appedix~\ref{sec:NormalGeo}.
\end{proof}
Most probable paths were previously studied as normal geodesics without the 
end-point condition $\chi(T) = 0$ and with the assumption of normality \cite{sommer_evolution_2015}. The corollary makes this assumption unnecessary and strengthens the characterization with the end-point condition.
One can consider the condition $\chi(T) =0$ as ensuring that our endpoint is the optimal point in $\pi^{-1}(y)$. For a simple analogue, one may consider the distance from a point to a line in $\mathbb{R}^2$, where the optimal path is a geodesic with the final condition that it must hit the line orthogonally.

\subsection{Representation in a parallel frame} \label{sec:RepParallel}
  We can write the system \eqref{eq:mpp} in a parallel frame as follows. This concrete form can be used for numerical integration of the system.

  Let $f \in OM_x$ be an arbitrary initial frame and define $f(t)$ by parallel transport. Write $f(t)^{-1} \Sigma f(t) = S^2$. Let $(S^{ij})$ be the inverse of $S$. Write
\begin{equation}
  \begin{split}
&v(t) = (v_1(t), \dots, v_d(t))^\dagger = f(t)^{-1} \dot \gamma(t)\\ &R_{ijkl}(t) = \langle f_l(t), R(f_i(t), f_j(t)) f_k(t) \rangle_g.
  \end{split}
\end{equation}
Finally, consider a matrix $(\chi_{ij}(t))$ in $\so(d)$ such that $\chi(t) = \frac{1}{2} \sum_{i,j=1}^d \chi_{ij}(t) f_j(t) \wedge f_i(t)$. The above equations take the form
\begin{equation} \left\{ \qquad
  \begin{split}
&\dot v_r(t)  =  \frac{1}{2} \sum_{i,j,k,l,s=1}^d S_{rs} S_{sl} R_{jikl}(t) \chi_{ij}(t) v_k(t), \\
&\dot \chi_{ij}(t) = \sum_{k=1}^d (v_j(t) S^{ik} - v_i(t) S^{jk}) S^{kl} v_l(t), \\
&\chi_{ij}(T) = 0.
  \end{split}
  \label{eq:mpp_coords} \right.
\end{equation}
The dependence on the parallel frame $f(t)$ is in the coefficients $R_{ijkl}(t)$, unless $\nabla R=0$, in which case these coefficients are constant. If we choose $f$ as eigenframe, then $S = \Lambda = \diag\{ \lambda_1, \dots, \lambda_d\}$ is a diagonal matrix, and the equations in \eqref{eq:mpp_coords} reduce to
\begin{equation} \dot v_l(t) =  \frac{\lambda_l^2}{2} \sum_{i,j,k=1}^d  R_{jikl}(t) \chi_{ij}(t) v_k(t), \quad
\dot \chi_{ij}(t)  = \frac{\lambda_j^2 - \lambda_i^2}{\lambda_i^2 \lambda_j^2} v_i(t) v_j(t).
\label{eq:mpp_coords_Eig}
\end{equation}

\section{Most probable paths on surfaces} \label{sec:surfaces}
\subsection{Equations for most probable paths}
We now explore the particular equations for the case when $d=2$. For $y \in M$, let $\kappa(y)$ be the Gaussian curvature of $M$ at $y \in M$.
\begin{corollary}
For a given $\Sigma \in \Sym^+ T_x M$, $x \in M$, let $\lambda_1 \geq \lambda_2 > 0$ be its eigenvalues. For a curve $\gamma: [0,T] \to M$ starting at $x$, define $f_1(t)$, $f_2(t)$ as a parallel eigenframe of $\Sigma_{\ptr_t}$ along $\gamma(t)$ such that $f_j(t)$ corresponds to $\lambda_j$.
Then $\gamma$ is the solution of
$$\dot \gamma(t) = c\lambda_1 \cos \theta(t) f_1(t) + c \lambda_2 \sin \theta(t) f_2(t),$$
with
\begin{equation}
\label{2dSolution} \dot \theta(t) = c^2 \kappa(\gamma(t)) h(t), \qquad \dot h(t) = -  \frac{\lambda_1^2 - \lambda_2^2}{2} \sin 2\theta(t), \qquad h(T) = 0.
\end{equation}
\end{corollary}

\begin{proof}
If we write $\gamma(t) = v_1(t) f_1(t) + v_2(t) f_2(t)$, then $c =\| \Sigma^{-1}_{\ptr_t} \dot \gamma\|_g = \sqrt{\lambda_1^{-2} v_1(t)^2 + \lambda_2^{-2} v_2(t)}$ is a first integral by Remark~\ref{re:Parametrization}. This allows us to write
$$\dot \gamma(t)  = v_1(t) f_1(t) + v_2 f_2(t) = c \lambda_1 \cos \theta(t) f_1(t) + c \lambda_2 \sin \theta(t) f_2(t).$$
Write $\chi(t) = \chi_{12}(t) f_2(t) \wedge f_1(t) = \frac{c^2}{\lambda_1 \lambda_2} h(t) f_2(t) \wedge f_1(t)$. We can now use equations~\eqref{eq:mpp_coords_Eig} for the computation and the fact that in dimension $2$ the only non-zero curvature terms are
$R_{1221}(t) = R_{2112}(t) = -R_{1212}(t) = - R_{2121}(t) = \kappa(\gamma(t))$.
Inserting the above expressions into \eqref{eq:mpp_coords_Eig} gives us
\begin{align*}
\dot v_1 & = -  c \lambda_1 \dot \theta \sin \theta = - \lambda_1^2 \cdot \kappa(\gamma) \cdot \frac{c^2}{\lambda_1 \lambda_2} h \cdot c \lambda_2 \sin \theta, \\
\dot v_2 & = c \lambda_2 \dot \theta \cos \theta = \lambda_2^2 \cdot \kappa(\gamma) \cdot \frac{c^2}{\lambda_1 \lambda_2} h \cdot c \lambda_1 \cos \theta, \\
\dot \chi_{12} & = - \frac{\lambda_2^2 - \lambda_1^2}{\lambda_1^2 \lambda_2^2} c^2 \lambda_1 \lambda_2 \cos \theta \sin \theta, 
\end{align*}
which simplify to \eqref{2dSolution}.
\end{proof}

\subsection{Example: Constant curvature surfaces} 
For any $- \frac{\pi}{2} \leq y \leq \frac{\pi}{2}$, $0 \leq k \leq 1$, we define the elliptic integral of first kind by
$$F(y, k) = \int_0^y \frac{ds}{\sqrt{1- k^2 \sin^2 s}},$$
with the complete version $K(k) = F(\frac{\pi}{2}, k)$. Correspondingly, we define the Jacobi elliptic sine function $\sn(y,k)$ by $\sn(F(y,k),k) = \sin y$. We define the corresponding delta amplitude as $\dn(y,k) = \sqrt{1- k^2 \sn(y,k)^2}$. 

The following result for constant curvature surfaces has the qualitative consequence of most probable paths bending towards the direction of highest eigenvalue $\lambda_1$ with positive curvature relative to a geodesic connecting the endpoint. For negative curvature, the curve bends towards the direction of lowest eigenvalue $\lambda_2$ with negative curvature. It was hypothesized in \cite{sommer_evolution_2015} that the behavior would be as in the positive curvature case. The result thus answer affirmative to this hypothesis, however only in the positive curvature case. See also Figure~\ref{fig:intro}.

\begin{theorem}
  Let $(M,g)$ be a two-dimensional manifold of constant Gaussian curvature $\kappa$. Let $\Sigma \in \Sym^+ T_x M$ be a chosen element with eigenvalues $\lambda_1 \geq \lambda_2 > 0$
  and with corresponding eigenvectors $f_1^x$, $f_2^x$.
  Let $\gamma$ be a most probable path with respect to $\Sigma$, normalized by $\| \Sigma_{\ptr_t}^{-1/2} \dot \gamma(t) \|_{g} =1$ and with $f_j(t) = \ptr_t f_j^x$. 
  If $\kappa>0$, the most probable paths will be of the form
\begin{align*}
\dot \gamma(t) & = \pm \lambda_1 \dn \left( K(k) + \alpha(T-t), k \right) f_1(t) + \lambda_2 k\sn \left( K(k) + \alpha (T-t), k \right) f_2(t) 
.
\end{align*}
  Conversely, if $\kappa<0$, most probable paths will take the form
\begin{align*}
\dot \gamma(t) & =   \lambda_1 k\sn \left( K(k) + \alpha (T-t), k \right) f_1(t)  \pm \lambda_2 \dn \left( K(k) + \alpha(T-t), k \right) f_2(t) .
\end{align*}
\end{theorem}

\begin{proof}
For $\alpha  > 0$ and a given initial value $- \pi < \psi_0 < \pi$, consider $\psi(t) = \psi(t; \psi_0, \alpha)$ as the solution of the non-linear pendulum equation
$$\ddot \psi(t) + \alpha^2 \sin \psi(t) =0, \qquad \psi(0) = \psi_0, \qquad \dot\psi(0) = 0.$$
It is classical that the solution of this equation is
$$\psi(t) = 2 \sin^{-1}\left( k \, \sn \left( K(k) - \alpha t, k \right) \right), \qquad k = \sin \frac{\psi_0}{2},$$
with period $2 \tau = \frac{4 K(k)}{\alpha}$. We note that $\dot \psi(t) = 0$ if and only if $t = n \tau$ for some integer $n$.
 
Assume now that $(M,g)$ is a two-dimensional manifold of constant Gaussian curvature $\kappa$. We can then rewrite the equations \eqref{2dSolution} as
\begin{equation} \label{pendulum} 2 \ddot \theta + (\lambda^2_1 - \lambda_2^2) \kappa \sin 2 \theta = 0, \qquad \dot \theta(T) = 0.\end{equation}
First consider $\kappa = \frac{1}{r^2} > 0$. Define $\alpha^2 = (\lambda_1^2 - \lambda_2^2) \kappa$. Then $2\theta(t)$ is the solution of the non-linear pendulum equation. By possibly replacing $f_1(t)$ with $-f_1(t)$, we may assume that $- \frac{\pi}{2}  \leq \theta(0) \leq \frac{\pi}{2}$. If $\theta(0) = \pm \frac{\pi}{2}$, the only solution is a constant solution. For $-\frac{\pi}{2} < \theta(0) < \frac{\pi}{2}$, it follows that $\theta(t) \in (- \frac{\pi}{2}, \frac{\pi}{2} )$ for all time. As a consequence, we must have that for some $\psi_0 \in (-\pi,\pi)$,
$$\textstyle \theta(t) = \frac{1}{2} \psi(t-T; \psi_0, \alpha) .$$
Remark that $\theta(t)$ has period
$$2\tau = \frac{4K(k)}{\alpha} = \frac{4 K(\sin \frac{\psi_0}{2} )}{(\lambda_1^2 - \lambda_2^2) \kappa}.$$
In summary, most probable paths will be of the form
\begin{equation}
  \begin{split}
\dot \gamma(t) & = v_1(t) f_1(t) + v_2(t) f_2(t) = \pm \lambda_1 \cos (\theta(t)) f_1(t) + \lambda_2 \sin (\theta(t)) f_2(t) \\
& = \pm \lambda_1 \dn \left( K(k) + \alpha(T-t), k \right) f_1(t)  \\
& \qquad + \lambda_2 k\sn \left( K(k) + \alpha (T-t), k \right) f_2(t) .
  \end{split}
  \label{eq:kp}
\end{equation}

If $\kappa = - \frac{1}{r^2} < 0$, then $\pi-2\theta$ solves the pendulum equation with $\alpha = \sqrt{(\lambda_1^2- \lambda_2^2) |\kappa|}$. We will then have similar results, with the difference that we now have oscillations in the direction of $f_2(t)$ rather than $f_1(t)$. In other words,
\begin{equation}
  \begin{split}
\dot \gamma(t) & = v_1(t) f_1(t) + v_2(t) f_2(t)  \\
& =   \lambda_1 k\sn \left( K(k) + \alpha (T-t), k \right) f_1(t) \\
& \qquad   \pm \lambda_2 \dn \left( K(k) + \alpha(T-t), k \right) f_2(t) .
  \end{split}
  \label{eq:km}
\end{equation}
\end{proof}

We can solve \eqref{eq:kp} in $\SO(3)$ if $M$ is identified as a subset of the sphere with radius $r = \kappa^{-1/2}$, centered at the origin. Let $q(t) \in \SO(3)$ be the solution of
$$q(t)^{-1} \dot q(t)  = \frac{1}{r} \begin{pmatrix} 0 & 0 & v_1(t)  \\
0 &0 & v_2(t) \\
- v_1(t) & -v_2(t) & 0 \end{pmatrix}, \qquad q(0) = 1_3.$$
If we consider $f_1(0)$, $f_2(0)$ and the initial point $x$ as elements in $\mathbb{R}^3$, the solution is given by
$$\gamma(t) = \begin{pmatrix} f_1(0) & f_2(0) & \frac{1}{r} x \end{pmatrix} q(t) \begin{pmatrix} 0 \\ 0 \\ r \end{pmatrix}.$$
Examples are visualized in the Figures~\ref{fig:S2_1}, \ref{fig:S2_2}, \ref{fig:S2_3}
numerical integration in $\SO(3)$. These are completed in MATLAB using a modification of the DiffMan package \cite{DiffMan}.

Similarly, in the $\kappa<0$ case, we can consider $M$ as a subset of $\mathbb{R}^{2,1}$ with $\langle a, a \rangle = - r^2$ and solve \eqref{eq:kp} in $\SO(2,1)$ with such that
$$q(t)^{-1} \dot q(t)  = \frac{1}{r} \begin{pmatrix} 0 & 0 & v_1(t)  \\
0 &0 & v_2(t) \\
v_1(t) & v_2(t) & 0 \end{pmatrix}, \qquad \gamma(t) = \begin{pmatrix} f_1(0) & f_2(0) & \frac{1}{r} x \end{pmatrix} q(t) \begin{pmatrix} 0 \\ 0 \\ \frac{1}{r} \end{pmatrix}.$$

\begin{figure}[h]
    \centering
    \begin{subfigure}[b]{\textwidth}
        \centering
        \includegraphics[width=0.4\linewidth]{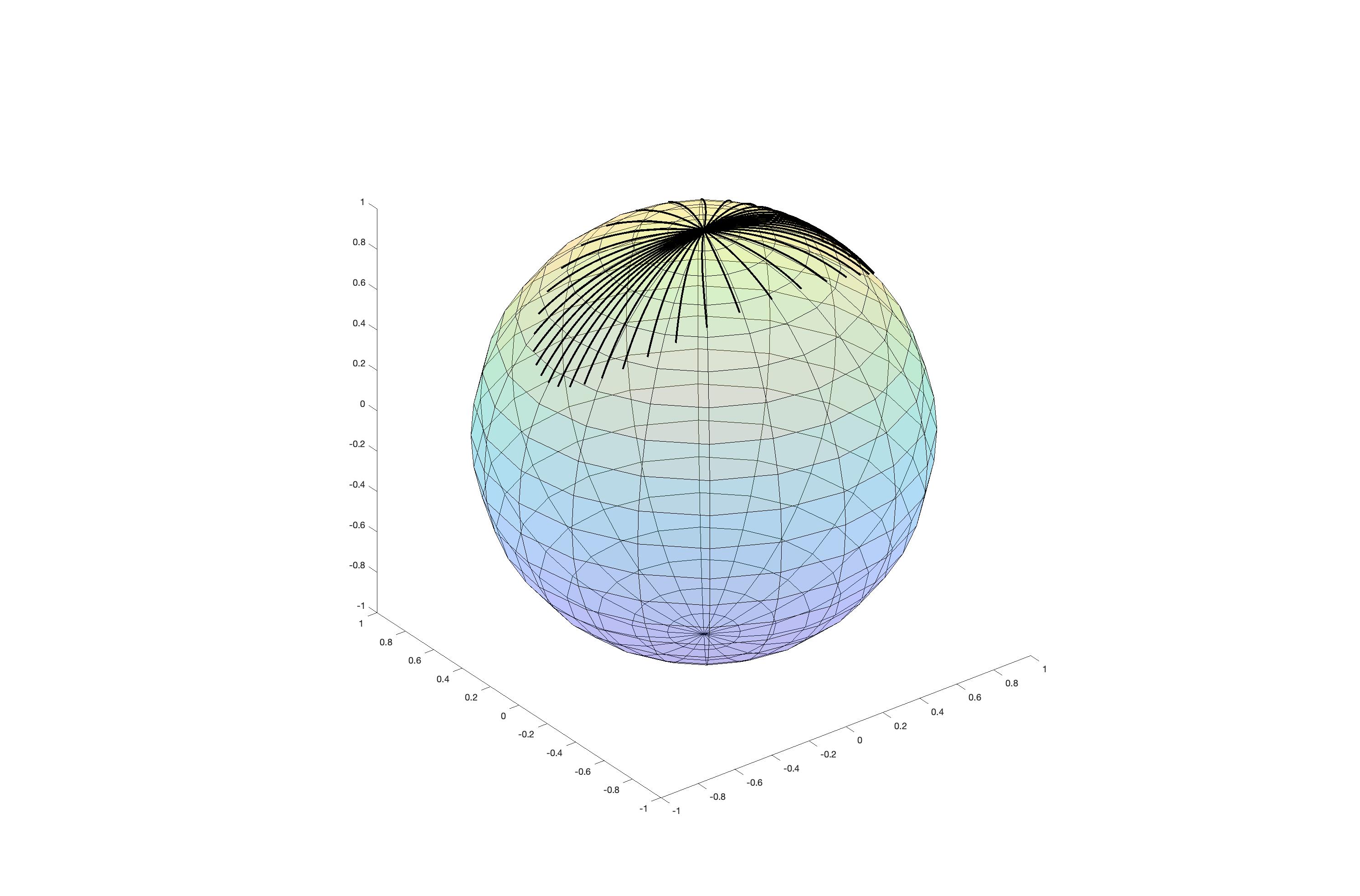} \hspace{-0.5cm}
        \includegraphics[width=0.35\linewidth]{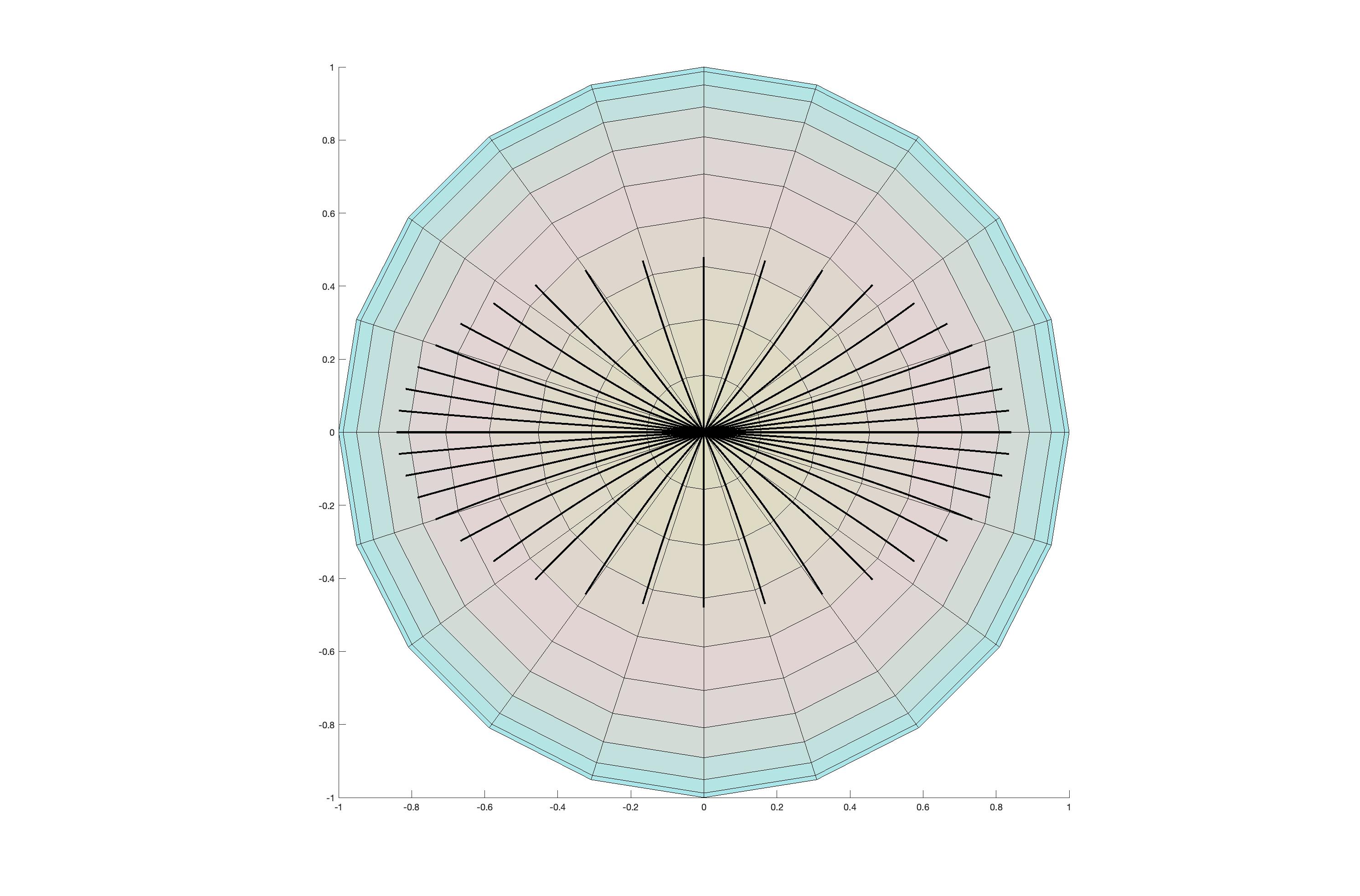} \hspace{-0.5cm}
    \end{subfigure}
    \caption{The figures above show examples of most probable paths on the unit sphere starting at the north pole with $\lambda_1 = 2$ and $\lambda_2 =1$. The two figures show examples of of most probable paths with $T =1/2$ seen from the side and above.}
    \label{fig:S2_1}
\end{figure}

\begin{figure}[h]
    \centering
    \begin{subfigure}[b]{\textwidth}
        \centering
        \includegraphics[width=0.45\linewidth]{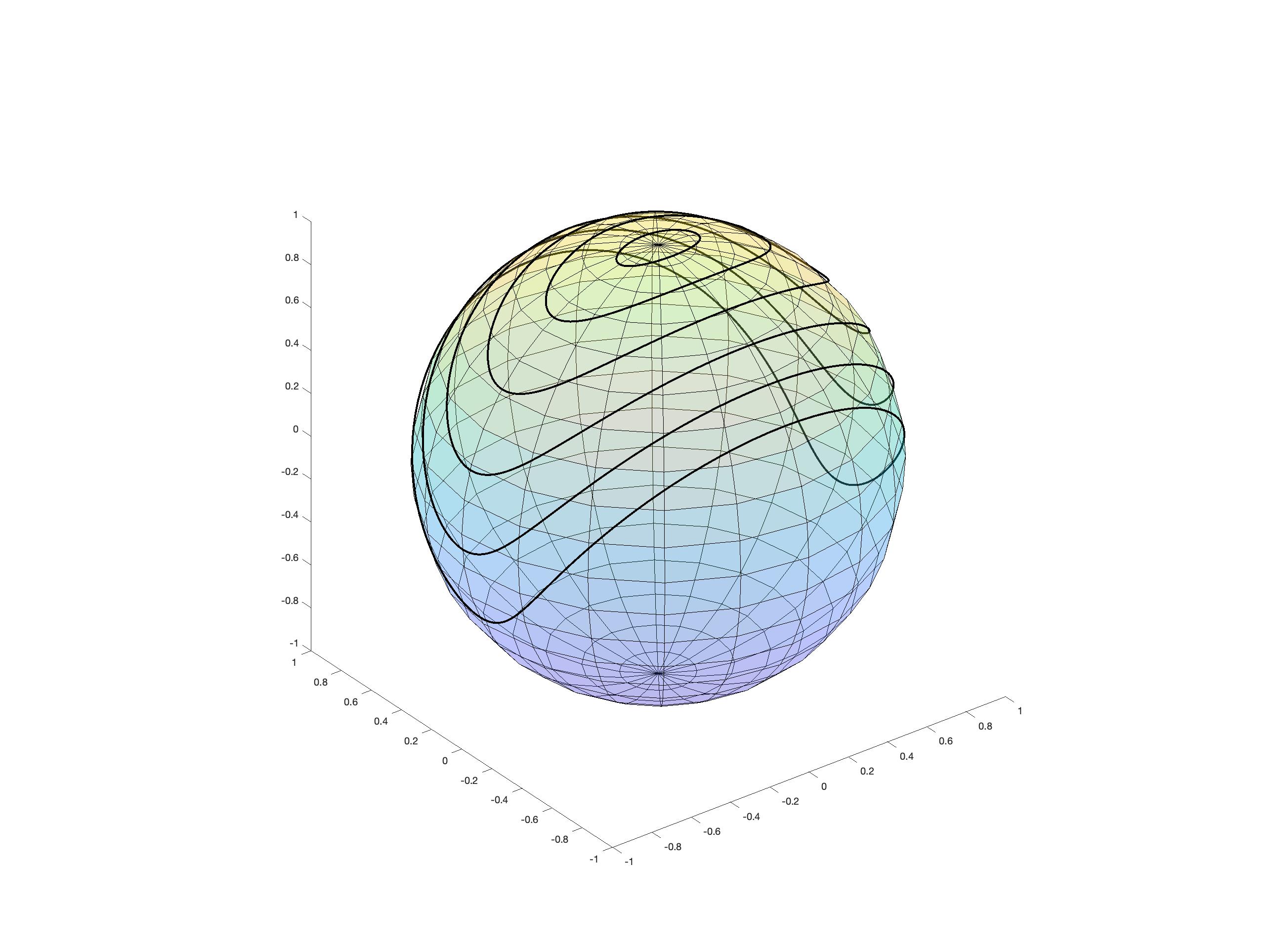} \hspace{-0.5cm}
        \includegraphics[width=0.35\linewidth]{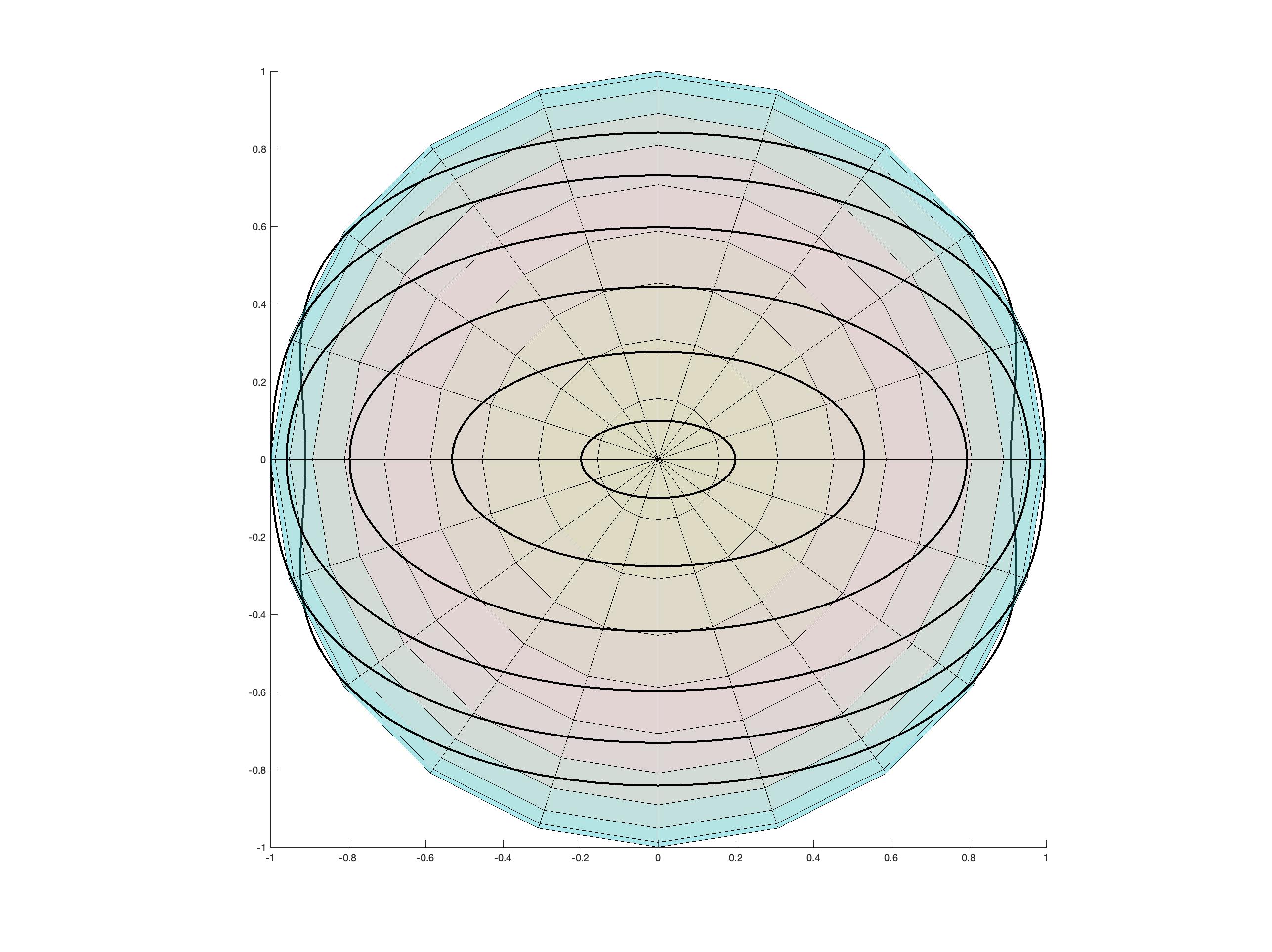}
    \end{subfigure}
    \caption{For the unit sphere with $\lambda_1 = 2$ and $\lambda_2 = 1$,  we have mapped the endpoints of the most probable paths for different values of $T$.}
    \label{fig:S2_2}
\end{figure}

\begin{figure}[h]
    \centering
    \begin{subfigure}[b]{\textwidth}
        \centering
        \includegraphics[width=0.4\linewidth]{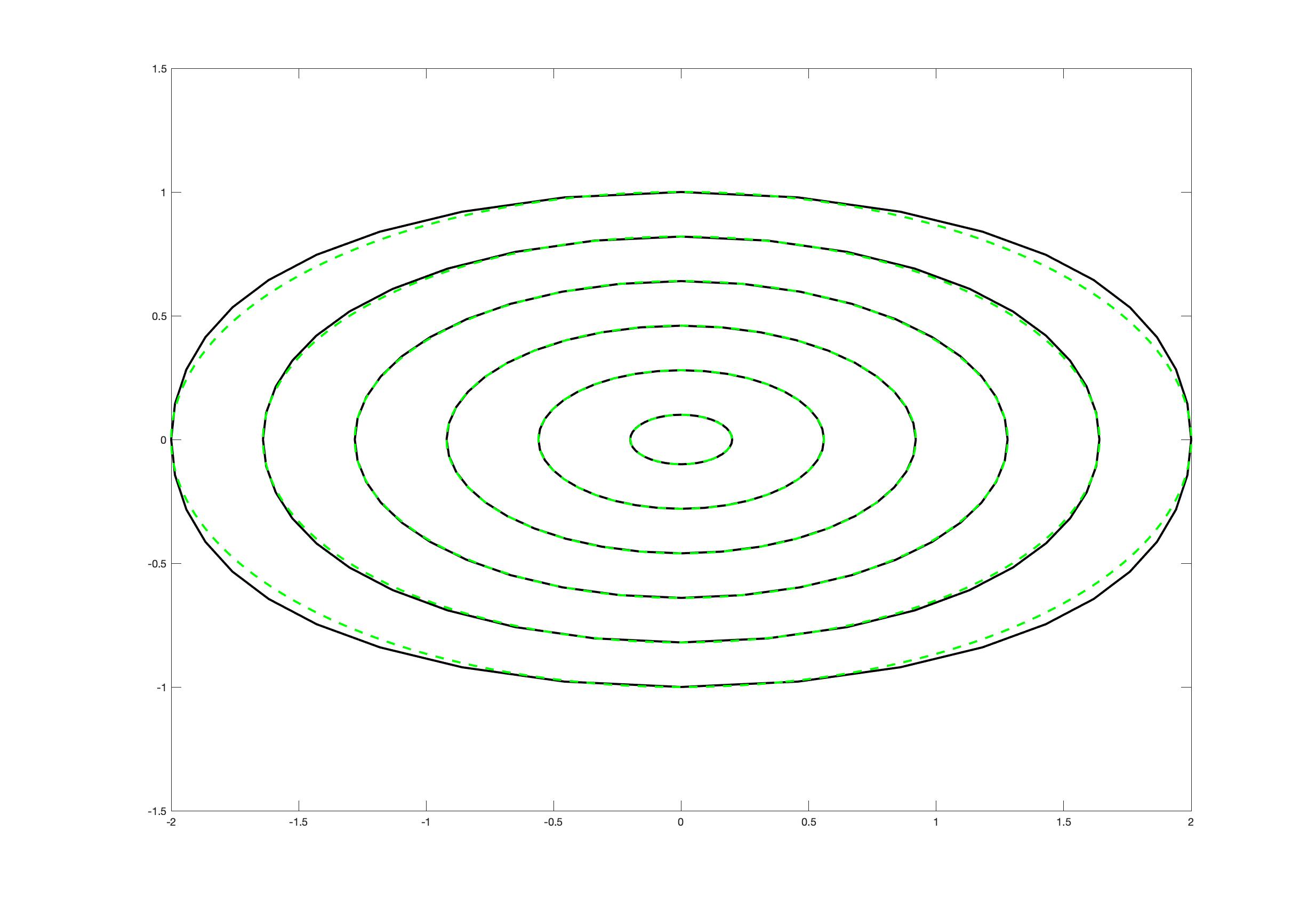}
        \includegraphics[width=0.45\linewidth]{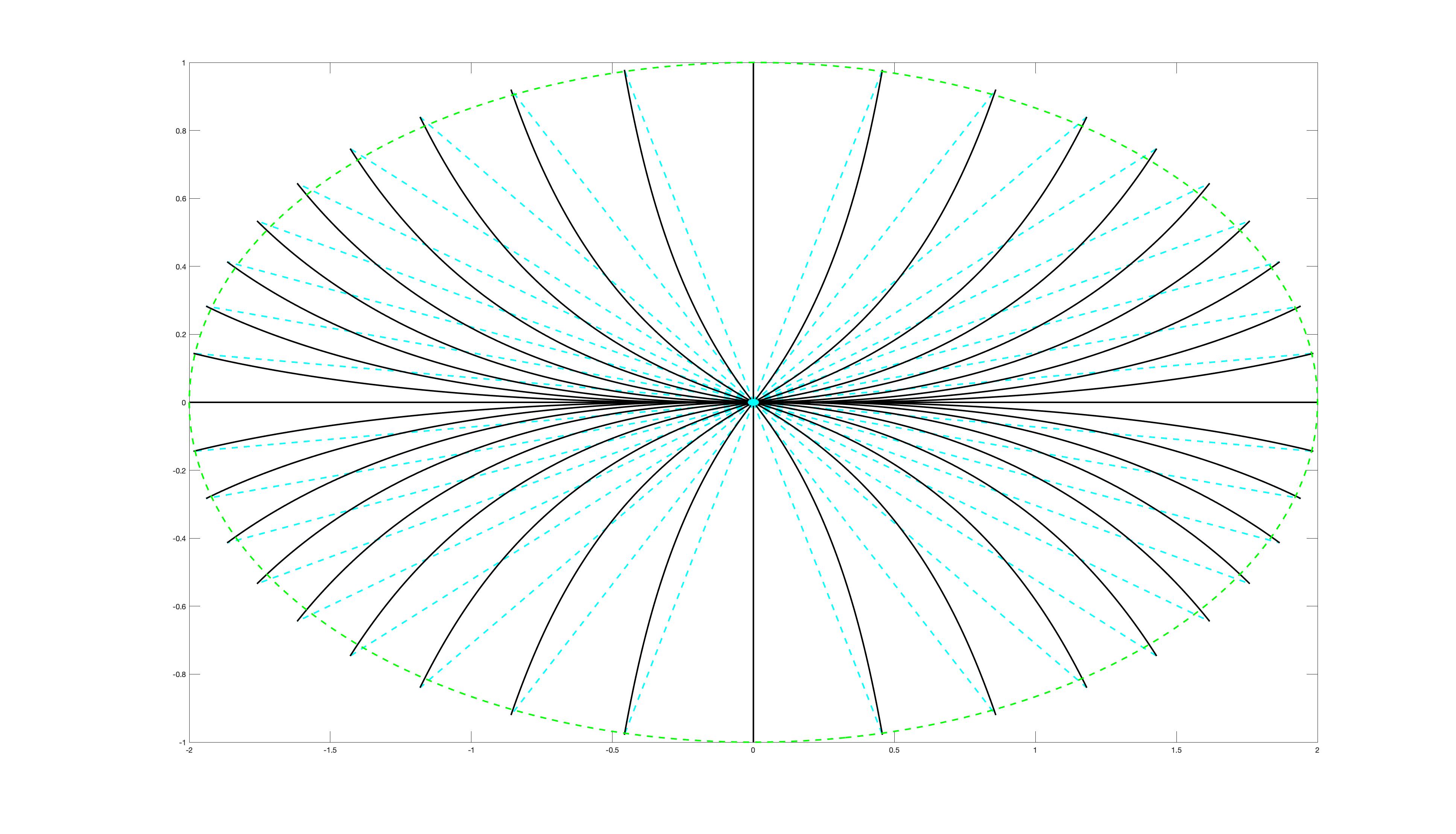} \hspace{-0.5cm}
    \end{subfigure}
    \caption{Again, we consider the unit sphere with $\lambda_1 = 2$ and $\lambda_2 = 1$, but now in a normal coordinate system centered at the north pole.
    The first image shows the endpoints compared to ellipses. The second shows most probable paths relative to geodesics.}
    \label{fig:S2_3}
\end{figure}

\section{Locally symmetric and symmetric spaces}\label{sec:examples}

Let now $(M,g)$ be a locally symmetric space, i.e. a space were the curvature  satisfies $\nabla R =0$. We consider a given mean $x \in M$ and covariance $\Sigma \in \Sym^+ T_x M$. Let $f^x$ be an orthonormal eigenframe of $\Sigma$ with corresponding eigenvalues given by the diagonal matrix $\Lambda^2 = \diag\{ \lambda_1^2, \dots, \lambda_d^2\}$. Let $G \subseteq OM$ be the subset of frames that can be obtained by parallel transport of $f^x$. Then there is some subgroup $H \subseteq \Ort(d)$ of elements $q$ such that $f^x \cdot q \in G$ and we have a principal bundle structure
$$H \to G \stackrel{\pi}{\to} M.$$
We note that $\underline{R} = \underline{R}(f)$ is constant for every $f \in G$, and by the Ambrose-Singer theorem, the image of $\underline{R}$ spans the Lie algebra $\mathfrak{h} \subseteq \so(d)$ of $H$. Define a Lie algebra~$\mathfrak{g}$ as the vector space $\mathbb{R}^d \times \mathfrak{h} \subseteq \mathbb{R}^d \times \so(d)$ with Lie brackets
$$[(a, A), (b,B)]_{\mathfrak{g}} = (Ab - Ba, - \underline{R}(a,b) + [A, B]).$$
This is a well-defined Lie algebra since $[A, \underline{R}(a,b)]_{\mathfrak{g}} = \underline{R}(Aa, b) + \underline{R}(a, Ab)$ for any $a, b \in \mathbb{R}^d$ and $A \in \mathfrak{h}$. We will use this Lie group structure to give the following result.

\begin{corollary} \label{cor:LieEq}
Let $\gamma:[0,T] \to M$ be a curve starting at $x \in M$, with $v(t) = (f^x)^{-1} \ptr_t^{-1} \dot \gamma(t)$. Assume that $\gamma(t)$ is a most probable path. Then there is a curve $B(t)$ in $\mathfrak{h}$ such that
$$\frac{d}{dt} (\Lambda^{-2} v(t), B(t)) =- \left[(v(t), 0), (\Lambda^{-2}v(t) , B(t)) \right]_{\mathfrak{g}} , \qquad B(T) = 0.$$
\end{corollary}
\begin{proof}
In the notation of Section~\ref{sec:RepParallel}, if $B(t) =  \frac{1}{2} \sum_{i,j=1}^d \chi_{ij}(t) \underline{R}(e_j, e_i)$, then
$$\Lambda^{-2} \dot v(t) =B(t) v(t), \qquad \dot B(t) =  \underline{R}(v(t), \Lambda^{2} v(t)).$$
The result follows.
\end{proof}

Let us now continue to the case the structure of $\mathfrak{g}$ integrates to a Lie group structure on $G$ such that $M = G/H$ is a symmetric space. Define $\gamma: [0,T] \to M$ as the most probable path relative to a covariance $\Sigma \in \Sym^+ T_x M$. Let $\Lambda^2 = \diag\{ \lambda_1^2, \dots, \lambda_d^2\}$ be the eigenvalues of $\Sigma$ with an eigenframe $f^x$. Define $f: [0,T] \to G$ be the solution of $f(t)^{-1} \cdot \dot f(t) = (v(t), 0)$ and $f(0) = f^x$, i.e. the result of parallel transporting $f^x$. Define $\Phi(t) = (f^x)^{-1} \cdot f(t)$ and write $v_0 = (f^x)^{-1} \dot \gamma(0)$. Notice that
$$\Phi(t)^{-1} \cdot \dot \Phi(t) = (v(t), 0).$$
We then have the following result.
\begin{proposition}
Consider the maps $\pr_{\Lambda^{\pm 2}}: \mathfrak{g} \to \mathfrak{g}$ given by
$$\pr_{\Lambda^{\pm 2}} (a, A) = (\Lambda^{\pm2} a, 0), \qquad a \in \mathbb{R}^d, A \in \mathfrak{h}.$$
Then $\Phi(t)$ is a solution of
$$\dot \Phi(t) = \Phi(t) \pr_{\Lambda^2} \Ad(\Phi(t)^{-1} \Phi(T)) \pr_{\Lambda^{-2}} \Phi(T) \cdot \dot \Phi(T) .$$
\end{proposition}

\begin{proof}
 We then see that by Corollary~\ref{cor:LieEq},
$$\frac{d}{dt} \Ad(\Phi(t)) (\Lambda^{-2} v(t), B(t)) = (0, 0).$$
Hence, for some constant $(c, C) \in \mathfrak{g}$, we have
$$\Ad(\Phi(t)) (\Lambda^{-2} v(t), B(t)) =  (c,C).$$
Inserting $t=T$, we have
$$\Ad(\Phi(T))(\Lambda^{-2} v_T, 0 )= (c, C).$$
Using that $(v(t),0) = \Phi(t)^{-1} \cdot \dot \Phi(t)$, we have the result.
\end{proof}

\begin{example}
Consider the special case of $M = S^d$, where $G = \SO(d+1)$. In this case, $\mathfrak{g} = \so(d+1)$ and we write
$$\mathfrak{g} = \mathfrak{m} \oplus \mathfrak{h} = \left\{ \bar{A} = \begin{pmatrix} A & a \\ - a^\dagger & 0 \end{pmatrix} \, : \, A \in \so(d), a \in \mathbb{R}^d \right\},$$
where $\mathfrak{m}$ and $\mathfrak{h}$ correspond to respectively the cases when $A =0$ and $a = 0$. If we define
$$I = \begin{pmatrix}  1_d & 0 \\ 0 & 0 \end{pmatrix}, \qquad \bar{\Lambda} = \begin{pmatrix} \Lambda & 0 \\ 0 & 1 \end{pmatrix},$$
we then see that
$$\pr_{\Lambda^{\pm 2}} \bar{A}=  \bar{\Lambda}^{\pm 1} (\bar{A} - I  \bar{A} I) \bar{\Lambda}^{\pm 1}.$$
The equation we have to solve is then
\begin{align*}
\dot \Phi(t) & = \Phi(t) \Lambda (\Phi(t)^{-1} \bar{B} \Phi(t) - I \Phi(t)^{-1} \bar{B} \Phi(t) I) \Lambda,
\end{align*}
with $\bar{B} = \Phi(T) \Lambda^{-1} ( \Phi(T) \cdot \dot \Phi(T) - I \Phi(T) \cdot \dot \Phi(T) I) \bar{\Lambda}^{-1}  \Phi(T)^{-1}$.
\end{example}

\section{Algorithms}
\label{sec:algorithms}
We here describe algorithms for computing mean and covariance on both $S^2$ and more general cases. On $S^2$ we obtain a very efficient approximate solution, and on general manifolds using constrained optimization. In addition, we describe strategies for numerically integrating the most probable path dynamical equation, and how to optimize over those using automatic differentiation.

\subsection{Mean and covariance on $S^2$} We show a particular case for constructing an algorithm for the unit sphere $S^2$.
Let $y_1, \dots, y_n$ be i.i.d. samples on $M$. We want to find $\Sigma = C\Sigma'$ that minimize \eqref{eq:sample_est}.
We only need to find $\Sigma'$ and then can determine $C$ by \eqref{C}. 
\begin{enumerate}[\rm (1)]
\item Using a $2$ to $1$ surjection, we can see the set $\{ \tilde \Sigma \in \Sym^+ TS^2 \, : \, \det \tilde \Sigma =1 \}$ as the image of $\mathbb{R}_{\geq 0} \times \SO(3)$. We do this by associating each element $(a, q)$, $a\geq 0$, $q=(q_1, q_2, q_3)$ with the symmetric map $\tilde \Sigma \in \Sym^+ T_{q_3} S^2$  such that
$$\tilde \Sigma q_1 = e^{2a} q_1, \qquad \tilde \Sigma q_2 = e^{-2a} q_2.$$
\item We generate a `lattice' of comparison points on $S^2$. This points set only has to be generated once, and can then be reused for any dataset.
Choose $a_{\max} >0$ and $n_a \in \mathbb{N}$ and define
$$a_l = \frac{l}{n_a} a_{\max}, \qquad \alpha_l = \sqrt{e^{2a_l} -e^{-2a_l}}, \qquad T_{\max,l} = 2\pi e^{a},  l=0,1, \dots, n_a.$$
The value $T_{\max,l}$ is chosen so that the geodesic going along the eigenvector of $e^{-a}$ from the north pole has time to reach the south pole.
Next, for a chosen $n_{\psi}, n_T \in \mathbb{N}$, and for $i= 1, \dots, n_\psi$, $j = 1, \dots, n_R$, define
$$T_{j,l} = \frac{j}{n_T} T_{\max,l}, \qquad \psi_i = \frac{i}{n_\psi} \pi, \qquad k_i = \sin^{-1} (\psi_i/2).$$
Finally, we define
\begin{align*}
v_{i,j,l}(t) &= e^{a_l} \dn \left( K(k_i) + \alpha_l(T_{j,l}-t), k_i \right), \\
w_{i,j,l}(t) & =  e^{-a_l} k_i \sn \left( K(k_i) + \alpha_l (T_{j,l}-t), k_i  \right).
\end{align*}

Let $N = (0,0,1)^\dagger$ be the north pole and let $\Sigma'_l \in \Sym^+ T_N S^2$ be the symmetric endomorphism with eigenvalue $e^{2a_l}$ and $e^{-2a_l}$ in respectively the directions $(1,0,0)^\dagger$ and $(0,1,0)^\dagger$.
Define $q(t) = q_{i,j,l}(t)$ in $\SO(3)$ as solutions of
$$q(t)^{-1} \dot q(t) =\begin{pmatrix} 0 & 0 & v_{i,j,l}(t) \\ 0 & 0 & w_{i,j,l}  \\ - v_{i,j,l}(t) & - w_{i,j,l}(t) & 0 \end{pmatrix}, \qquad q(0) = 1_3 .$$
Finally, define $z_{i,j,l} = q_{i,j,l}(T_{j,l}) N$ and for $i >0$, define its mirror in the $y,z$-axis,
$z_{\pm i, \pm j,l} = \diag\{ \pm 1, \pm 1, 1\} z_{i,j,l}$. Then $z_{i,j,l}$ are all endpoints of most probable paths with length $T_{j,l}$.

In summary, we need to solve $(n_a +1) \times n_\psi \times n_T$-ODEs in $\SO(3)$. We will use the data $z_{i,j,l}$, $T_{j,l}$, $a_l$ in what follows.
\item We use the previous data to make an approximation to $d_\rho(\pi^{-1}(y), (N, \Sigma_l))$. We have a bound
$$d(\pi(y),(N,\Sigma_l)) \leq T_{j,l} + e^{a_l} \cos^{-1}(y^\dagger z_{i,j,l}),$$
from the fact that the most probable path in the direction of the eigenvalue $e^{-a_l}$ is the slowest moving in the Riemannian metric.
Define a function
$$\mathrm{Dist}_l(y) = \min_{\begin{subarray}{c} i = 0, \pm 1, \dots, \pm m_1 \\ j=1, \dots, m_2 \end{subarray}} \left( T_{j,l} + e^{a_k} \cos^{-1}(y^\dagger z_{i,j,l}) \right), \qquad y \in S^2.$$
\item Finally, find we define $a_{l'} \times q' \in \mathbb{R} \times \Ort(3)$ as the element corresponding to the best choice $\Sigma'$ by
$$l' \times q' = \argmin_{\begin{subarray}{c} l=0,1, \dots, n_a \\ q \in \Ort(3) \end{subarray}} \sum_{r=1}^n \mathrm{Dist}_l(q^{-1} y_r)^2.$$
This can be done by using an optimizer in $q$ or optimizing over a grid $n_1 \times n_2 \times n_3$ of $\{ q_{i,j,k}' \, : \,  i =1, \dots, n_1, j=1, \dots, n_2, k=1,\dots, n_3\}$, where
\begin{multline}
q_{i,j,k}' = \begin{pmatrix} \cos( \frac{2\pi i}{n_1-1}) & \sin( \frac{2\pi i}{n_1-1}) & 0 \\ - \sin( \frac{2\pi i}{n_1-1}) & \cos( \frac{2\pi i}{n_1-1}) & 0 \\ 0 & 0 & 1 \end{pmatrix} 
\begin{pmatrix} \cos( \frac{2\pi i}{n_2-1}) & 0 & \sin( \frac{2\pi i}{n_2-1})  \\ 0 & 0 & 0 \\ - \sin( \frac{2\pi i}{n_2-1}) & 0  & \cos( \frac{2\pi i}{n_2-1})  \end{pmatrix} \\
\begin{pmatrix} 1 & 0 & 0 \\ 0 & \cos( \frac{2\pi i}{n_3-1}) & \sin( \frac{2\pi i}{n_3-1})  \\ 0 & - \sin( \frac{2\pi i}{n_3-1}) & \cos( \frac{2\pi i}{n_3-1})  \end{pmatrix}.
\end{multline}
\end{enumerate}

\subsection{General geometries}
For general Riemannian manifolds, the system \eqref{eq:mpp_coords} can be solved numerically by integrating $v(t),f(t)$ and $\chi(t)$ forward and solving for $\chi(T)=0$. Algorithm~\ref{alg:mpp} shows a simple gradient-based approach for finding most probable paths from a starting point $\Sigma\in \Sym^+ T_x M$ to $y\in M$.

\begin{algorithm} \DontPrintSemicolon \SetAlgoLined
  \KwData{$S\in \Sym^+\mathbb R^d$, $y\in M$, $f^x\in O(TM),\ T,\delta>0$}
  \KwResult{$v(0),\chi(0)$ s.t. $d(\gamma(T),y)^2+\|\chi(T)\|^2\le\delta$}
  \While{$d(\gamma(T),y)^2+\|\chi(T)\|^2>\delta$} {
    1. numerically integrate forward $\gamma(t)$ and $\chi(t)$\\
    2. compute gradient $g(v(0),\chi(0))\gets \nabla_{v(0),\chi(0)}\big(d(\gamma(T),y)^2+ \|\chi(T)\|^2\big)$\\
    3. update initial conditions: $(v(0),\chi(0))\gets (v(0),\chi(0))-\ve g(v(0),\chi(0))$
  }
  \caption{Most probable path from $\Sigma$ to $y$}
  \label{alg:mpp}
\end{algorithm} 
The deviation $d(\gamma(T),y)$ between the endpoint $\gamma(T)$ and the target $y$ can be replaced by, for example, the Euclidean distance in a chart or using an embedding of $M$.

The gradients of $d(\gamma(T),v)^2$ and $\|\chi(T)\|^2$ with respect to the initial conditions $v(0)$ and $\chi(0)$ can be derived by solving the adjoint of \eqref{eq:mpp_coords}. This can be achieved directly using automatic differentiation frameworks that implement the adjoint equations implicitly with reverse automatic differentiation. 
In practice, the gradient-descend algorithm above can be replaced by quasi-Newton methods such as BFGS to improve convergence.

Figure~\ref{fig:intro} shows examples of most probable paths on sphere $S^2$ and the hyperbolic space $H^2$ computed using Algorithm~\ref{alg:mpp}, and with derivatives computed using the Jax automatic differentiation framework \cite{jax2018github}, implemented in the JaxGeometry\footnote{\url{https://bitbucket.org/stefansommer/jaxgeometry}} package.


%

\subsection{Mean and covariance estimation}
\label{sec:statistics}
Given i.i.d. samples $y_1,\dots,y_n$, the mean and covariance estimator \eqref{eq:sample_est} can in general be found by solving the constrained minimization problem
\begin{equation}
  \begin{split}
  &
  \argmin_{\Sigma \in \Sym^+ TM, (v_1(0),\chi_1(0)),\dots,(v_n(0),\chi_n(0))}  \sum_{j=1}^n \left(v_i(0)^TS^{-2}v_i(0)+\ln\det_g\Sigma\right)
  \\
  &\mathrm{s.t.}\ 
  (\gamma(T)=y_1, \chi_1(T)=0),\dots,(\gamma(T)=y_n, \chi_n(T)=0)
  \end{split}
  \label{eq:constrained}
\end{equation}
where $\gamma_i(t),\chi_i(t)$ are the trajectories defined by \eqref{eq:mpp} with initial conditions $v_i(0),\chi_i(0)$. 

Let $F:\Sym^+ TM\times(\mathbb R^d\times\wedge^2\mathbb R^d)^{\times n}\to\mathbb R$ denote the objective function of \eqref{eq:constrained}.
The velocities $v_i(0)$ on which $F$ is evaluated in \eqref{eq:constrained} depend on $\Sigma$. Let $G:\Sym^+ TM\times\mathbb R^d\times\wedge^2\mathbb R^d\to\mathbb R^d\times \wedge^2\mathbb R^d$ denote a map $\Sigma,v,\chi\mapsto\begin{pmatrix}\gamma(T)-v\\\chi(T)\end{pmatrix}$ encoding the constraints such that $G(\Sigma,v_i(0),\chi_i(0))=0$, for example using a chart around $v$ to express the end-point difference $\gamma(T)-v$. The inverse function theorem implies that
\begin{equation}
  D_\Sigma v_i(0)
  =
  -\big(D_{v,\chi}G|_{\Sigma,v_i(0),\chi_i(0)}\big)^{-1}D_\Sigma G|_{\Sigma,v_i(0),\chi_i(0)}
  .
  \label{eq:Dvi}
\end{equation}
Thus, an infinitesimal change $\delta\Sigma$ of $\Sigma$ results in the variation 
\begin{equation}
  \begin{split}
  \delta F
  &=
  \nabla_\Sigma F\delta\Sigma+\sum_{i=1}^n\nabla_{v_i}FD_\Sigma v_i(0)\delta\Sigma
  \\
  &=
  \nabla_\Sigma F\delta\Sigma-\sum_{i=1}^n\nabla_{v_i}F\big(D_{v,\chi}G|_{\Sigma,v_i(0),\chi_i(0)}\big)^{-1}D_\Sigma G|_{\Sigma,v_i(0),\chi_i(0)}\delta\Sigma
  .
  \label{}
  \end{split}
\end{equation}
This leads to the iterative procedure for solving \eqref{eq:constrained} listed in Algorithm~\ref{alg:mpp_mean}.
\begin{algorithm} \DontPrintSemicolon \SetAlgoLined
  \KwData{$y_1,\dots,y_n\in M$, $f^x\in O(TM),\ T,\delta>0$}
  \KwResult{$\Sigma$ (locally) minimizing \eqref{eq:constrained}} 
  \While{$\max(\|\nabla_\Sigma F\|,\|\nabla_{(v_1,\chi_1)} \|G\|^2\|,\dots,\|\nabla_{(v_n,\chi_n)} \|G\|^2\|)>\delta$} {
    1. for $i=1,\dots,n$, numerically integrate forward $\gamma_i(t)$ and $\chi_i(t)$\\
    2. compute $\nabla_\Sigma F$, $D_{(v_i,\chi_i)} G$, $\|\nabla_{(v_i,\chi_i)} \|G\|^2\|$\\
    3. update $\Sigma$: $\Sigma\gets \Sigma-\ve (\nabla_\Sigma F+\sum_{i=1}^n\nabla_{v_i}FD_\Sigma v_i(0))$\\
    4. update $(v_i(0),\chi_i(0))$: $(v_i(0),\chi_i(0))\gets (v_i(0),\chi_i(0))-\ve \nabla_{(v_i,\chi_i)} \|G\|^2$
  }
  \caption{Most likely $\Sigma$ for observation $y_1,\dots,y_n$.}
  \label{alg:mpp_mean}
\end{algorithm} 

The right-hand side of \eqref{eq:Dvi} is in Algorithm~\ref{alg:mpp_mean} evaluated at the current guess for $(v_i(0),\chi_i(0))$ and hence provides only an approximation to the true derivative. This implies that the stability of the algorithm can increase by taking multiple update steps for the initial conditions $(v_i(0),\chi_i(0))$ for each update step for $\Sigma$. The convergence rate can in addition be increased by using e.g. descent-schemes with momentum such as the ADAM optimizer instead of pure gradient descent.

Figure~\ref{fig:estimation} show examples of estimation of $\Sigma$ on the sphere $S^2$, $T^2$ and $H^2$ using Algorithm~\ref{alg:mpp_mean}.
\begin{figure}[h]
    \centering
    \begin{subfigure}[b]{\textwidth}
        \centering
        \includegraphics[width=0.28\linewidth,trim=200 250 200 200,clip]{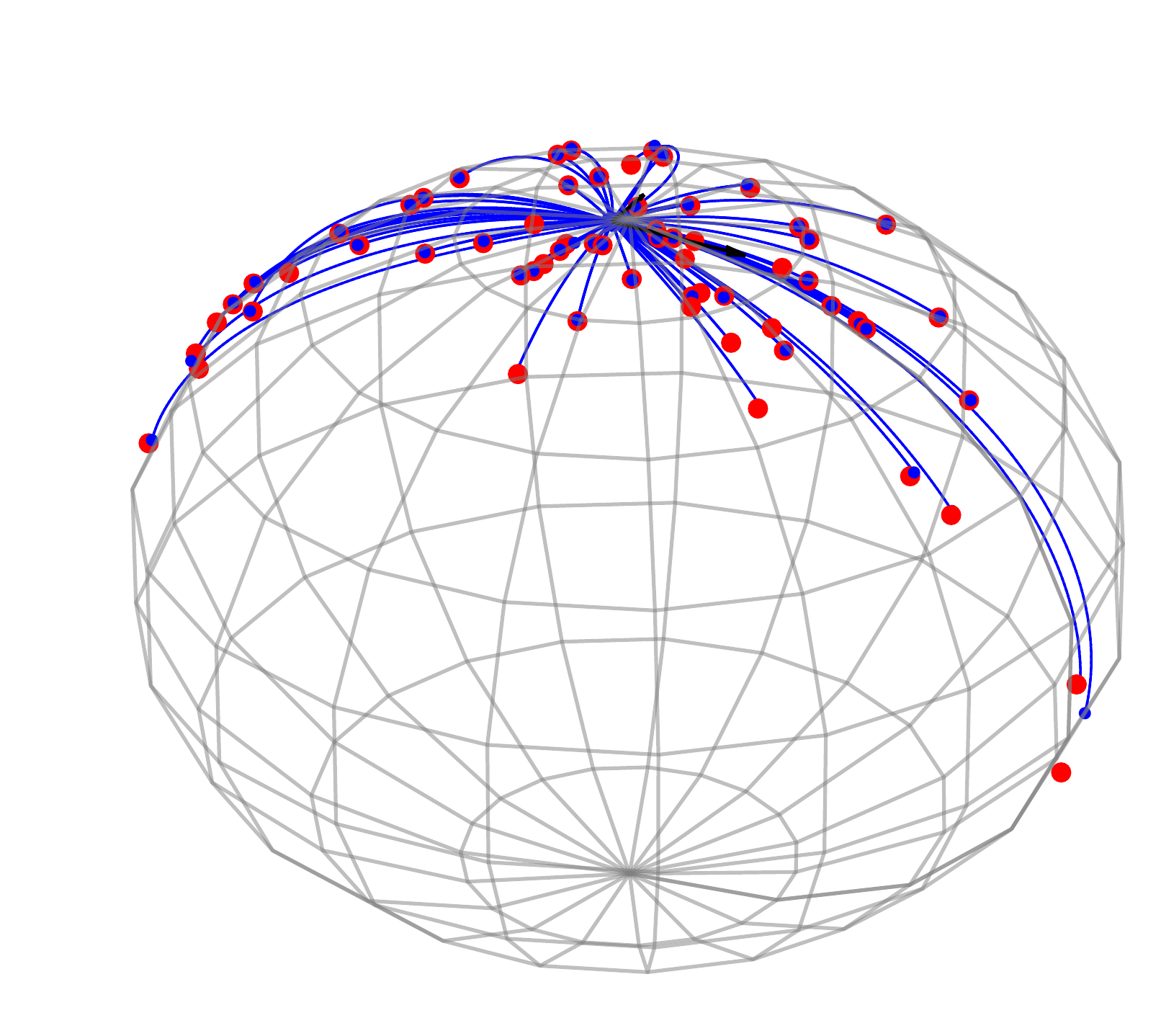}
        \includegraphics[width=0.37\linewidth,trim=150 250 100 200,clip]{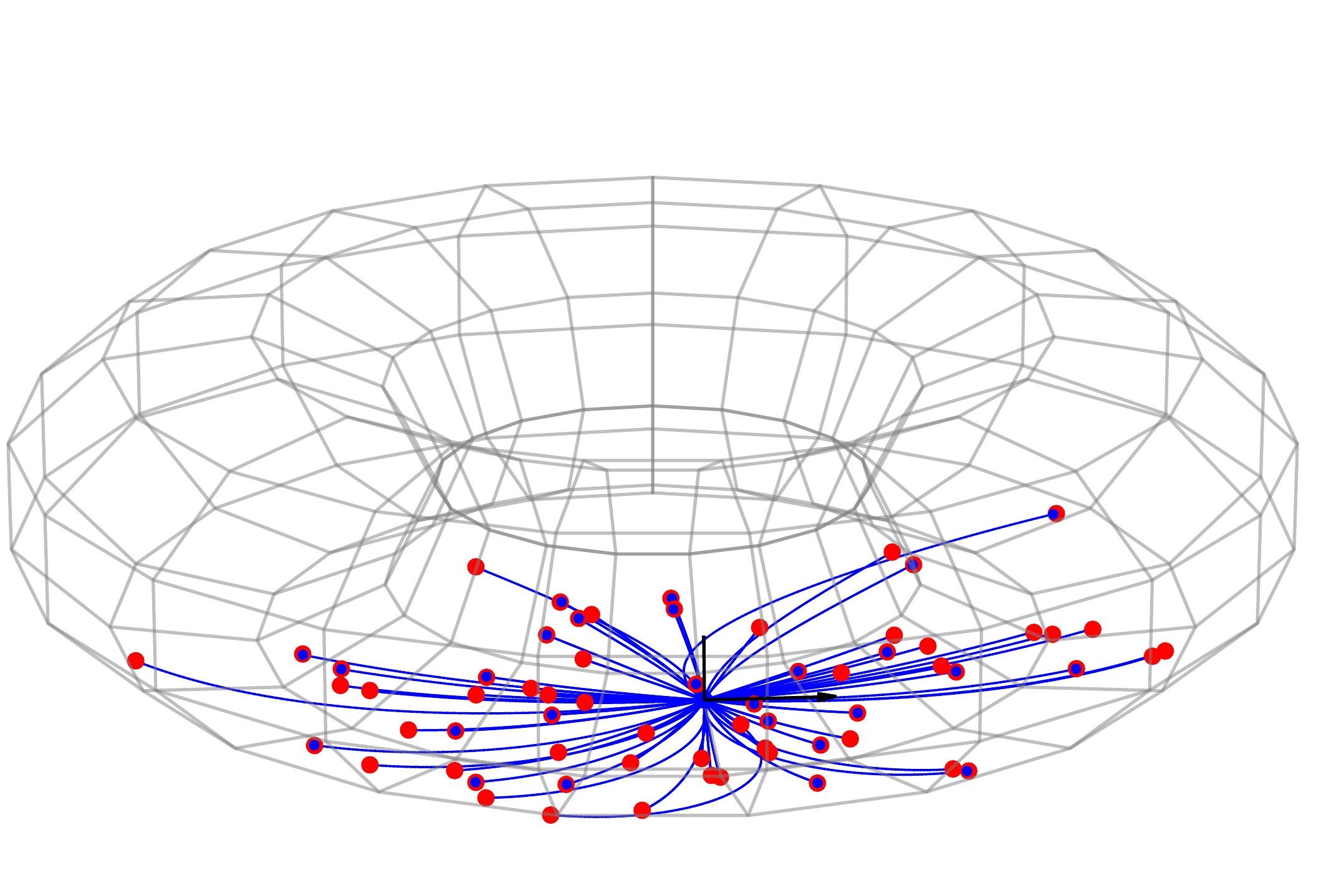}
        \includegraphics[width=0.27\linewidth,trim=170 350 230 100,clip]{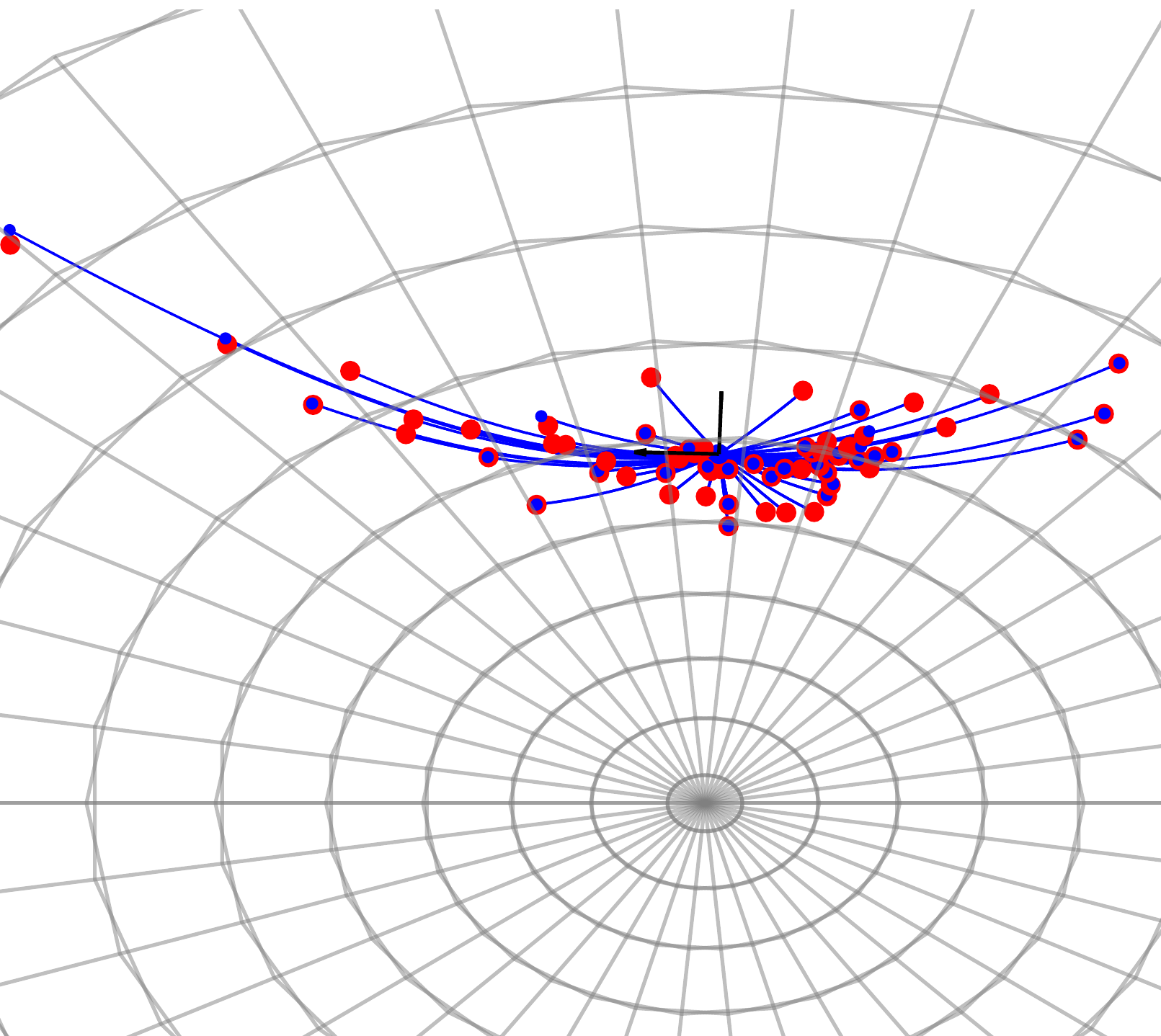}
    \end{subfigure}
    \caption{Mean and covariance estimation on the sphere $S^2$ and the embedded torus $T^2$ using Algorithm~\ref{alg:mpp_mean}. Blue curves shows MPPs from the estimated mean to the 64 samples (red points).}
    \label{fig:estimation}
\end{figure}

\appendix

\section{Definition of sub-Riemannian geometry} \label{sec:SR}
We give a quick introduction to sub-Riemannian geometry and refer to \cite{Mon02} for details. A sub-Riemannian manifold is a triple $(M, E,\rho)$ where $M$ is a connected manifold, $E$ is a subbundle of the tangent bundle $TM$ and $\rho = \langle \cdot , \cdot \rangle_\rho$ is a metric tensor defined only on $E$. This tensor defines a vector bundle morphism $\sharp^\rho: T^*M \to E \subseteq TM$ given by
$$\alpha(v) = \langle \sharp^\rho \alpha, v \rangle_\rho.$$
Consequently, we obtain a positive semi-definite symmetric tensor $\rho^* = \langle \cdot , \cdot \rangle_{\rho^*}$ on $T^*M$ defined by
$$\langle \alpha, \beta \rangle_{\rho^*} = \langle \sharp^\rho \alpha, \sharp^\rho \beta \rangle_\rho.$$
This tensor degenerates along the subbundles $\Ann(E) \subseteq T^*M$ of covectors vanishing on~$E$. It follows that a sub-Riemannian manifold can equivalently be defined as a connected manifold with positive, semi-definite cometric $\rho^*$ that degenerates along a subbundle.

An absolutely continuous curve $\gamma:[0,T] \to M$ is called \emph{horizontal} if $\dot \gamma(t) \in E_{\gamma(t)}$ for almost every $t$.
For such a curve, we define its length to be
$$L^\rho(\gamma) = \int_0^T | \dot \gamma|_\rho(t)\, dt.$$
This length is invariant under reparametrization, so we can restrict our considerations to the case $T =1$.

For any $x, y \in M$, we define
$$d_\rho(x,y) = \inf \left\{ L^\rho(\gamma)  \, : \, \begin{array}{c} \text{$\gamma:[0,1] \to M$ horizontal} \\ \gamma(0) = x, \gamma(1) = y \end{array} \right\}.$$
We notice that if there are no horizontal curves connecting the two points, then $d_\rho(x,y) =\infty$. Let $x \in M$ be a given point. We define $\mathcal{C}(x)$ as the space of all horizontal curves defined on $[0,1]$, with $L^2$-derivative, that start in $x$. This collection has a natural structure of a Hilbert manifold, see \cite[Chapter~5.1]{Mon02} for details. Define a mapping
$$\Pi: \mathcal{C}(x) \to M, \qquad \gamma \mapsto \gamma(1).$$
Define $\mathcal{C}(x,y) = \Pi^{-1}(y)$. A point $\gamma \in \mathcal{C}(x,y)$ is called \emph{regular} if $\Pi_{*,\gamma}:T_\gamma \mathcal{C}(x) \to T_y M$ is surjective. Otherwise, $\gamma$ is called \emph{singular} or \emph{abnormal} curves.

Assume that $\mathcal{C}(x,y)$ non-empty. Define $F: \mathcal{C}(x,y) \to \mathbb{R}$ by $\gamma \mapsto L^\rho(\gamma)$. We look at minimal elements in $\mathcal{C}(x,y)$ with respect to $F$. If $\gamma$ is a regular curve, then $\mathcal{C}(x,y)$ locally has the structure of a Hilbert manifold around $\gamma$ by the inverse function theorem. Hence, any regular minimal element must be a critical, i.e. we must have $F_{*,\gamma} = 0$. Such curves are called \emph{normal geodesics}, and will always be locally length minimizing. It can be shown that all such curves, up to reparametrization, be found as a projecting of a solution of a Hamiltonian system. The Hamiltonian is given by
$$P(\alpha) = \frac{1}{2} \langle \alpha, \alpha\rangle_{\rho^*}, \qquad \alpha \in T^*M.$$
In conclusion, length minimizers are either normal geodesics or abnormal curves. These classes of curves are not necessarily disjoint.

We say that $E$ is bracket-generating if for every point $x \in M$,
$$\spn \{ X_i, [X_i,X_j], [X_i, [X_j, X_k]], \dots, \} |_x \in T_xM, \qquad X_i \in \Gamma(E),$$
that is, if sections of $E$ generate the entire tangent bundle $TM$. If this condition holds, then any pair of points can be connected by a horizontal curve. The value of $d_\rho$ is always finite, and furthermore, it induces the same topology as the manifold topology.

\begin{remark}
Let $L$ be a second order operator on $M$ without constant term, such that for any pair of smooth functions $f,g \in C^\infty(M)$,
$$L(fg) - fLg - gL f = \langle df, dg \rangle_{\rho^*}.$$
In other words, locally, $L$ can always be written as $L = \sum_{j=1}^{\rank E} V_i^2 + V_0$, where $V_1, \dots, V_{\rank E}$ is a local orthonormal basis of $(E,\rho)$. If $E$ is bracket generating, then $L$ is hypoelliptic~\cite{Hor67} and its heat semigroup $p_t(x;y)$ has a strictly positive density~\cite{StVa72}.
\end{remark}

\section{Sub-Riemannian normal geodesics on $(\Sym^+ TM, E, \rho)$} \label{sec:NormalGeo}
By the discussion in Section~\ref{sec:Alternative}, it follows that we can write a normal geodesic as $\Sigma(t) = f(t)^{-1} S^2 f(t)$ where $f(t)$ is a normal geodesic in $(\Ort(TM), \calH, \rho_S)$. We do the computations here.

Recall the definition of the vector fields $H_a$, $a \in \mathbb{R}^n$ and $\xi_A$, $A \in \so(n)$ in Section~\ref{sec:GeometryFrame}. We introduce corresponding Hamiltonian functions
$$P_a(\alpha) = \alpha(H_a), \qquad Q_A(\alpha) = \alpha(\xi_A), \qquad \alpha \in T^* \Ort(TM).$$ 
Our formulas in \eqref{Hxibrackets}, then give corresponding relations in terms of Poisson brackets
$$\{ P_a, P_b\} = Q_{\underline{R}(a,b)}, \qquad \{ Q_A, P_a\} = - P_{Aa}, \qquad \{ Q_A, Q_B\} = - Q_{[A,B]}.$$

Since $H_{Se_1}, \dots, H_{Se_d}$ is a global orthonormal basis, we have that the sub-Riemannian Hamiltonian is given by
$$P = \frac{1}{2} \sum_{j=1}^d P_{Se_j}^2.$$
Let $\lambda(t) = e^{t\vec{P}}(\lambda_0)$ be a solution in $T^* FM$ along $f(t)$ in $FM$ and define curves $v(t)$ in $\mathbb{R}^d$ and $A(t)$ in $\so(d)$ by
$$P_a(t) = P_a(\lambda(t)) = \langle S^{-2} v(t), a \rangle , \qquad Q_B(t) = Q_B(\lambda(t)) = - \langle A(t), B \rangle.$$ 
Then along a solution, we have
\begin{align*}
\langle S^{-2} \dot v, a \rangle & = \dot P_a = \{ P_a, P\} =  \sum_{j=1}^d P_{Se_j} Q_{\underline{R}(a, Se_j)} =  -  \sum_{j=1}^d \langle S^{-2} v, Se_j \rangle  \langle A, \underline{R}(a, Se_j) \rangle \\
& =  \sum_{j=1}^d \langle S^{-1} v, e_j \rangle  \langle \underline{R}(A)S e_j, a\rangle = \langle \underline{R}(A) v, a\rangle \\
- \langle \dot A(t) , B \rangle & = - \dot Q_B  = - \{ Q_B, P \} = \sum_{j=1}^d P_{Se_j} P_{BSe_j} =  \sum_{j=1}^d \langle S^{-2}v, Se_j \rangle \langle BS e_j ,S^{-2} v \rangle \\
& =  \langle B v, S^{-2}v \rangle = \langle v \wedge S^{-2} v, B \rangle.
\end{align*}
In summary, $\dot v = S^2 \underline{R}(A) v$ and $\dot A =  v \wedge S^{-2} v$.

Let $\beta \in \Gamma(T^* \Ort(TM))$ be a one-form on $\Ort(TM)$ and define the corresponding vertical lift $\vl \beta \in \Gamma(T(T^* \Ort(TM)))$ by
$$\vl \beta |_\alpha = \frac{d}{dt} (\alpha + t\beta_f) |_{t=0}, \qquad \alpha \in T_f^* \Ort(TM).$$  
Let $\vartheta$ be the Liouville one form $\vartheta|_\alpha = \pi^* \alpha$ with canonical symplectic form
$\sigma = - d\vartheta$. Observe that $\sigma(\vl \beta, \, \cdot \, ) = - (\pi^* \beta)(\, \cdot \,)$. Then
\begin{align*}
dP(\vl \beta) & = \sum_{j=1}^d P_{S e_j}  \beta(H_{S e_j} ) = \sigma( \vec{P}, \vl \beta) = \beta(\pi_* \vec{P}).
\end{align*}
If we consider this along the curve, we have
\begin{align*}
& \sum_{j=1}^d P_{S e_j} \beta(H_{S e_j} ) |_{\lambda(t)} = \sum_{j=1}^d \langle S^{-1}v(t),  e_j \beta(H_{S e_j} ) |_{f(t)} = \beta(H_{v(t)} ) |_{f(t)} \\
& = \beta(\pi_* \vec{P}) |_{\lambda(t)} = \beta(\dot f(t)) 
\end{align*}
It follows that
$$\dot f(t) = H_{v(t)}, \qquad \dot v(t) = S^2 \underline{R}(f(t))(A(t)) v(t), \qquad \dot A(t) =v(t) \wedge S^{-2} v(t).$$
We see that these are exactly the equations of found in the proof of Theorem~\ref{th:ProbPath} without the condition $A(T) = 0$.
\bibliographystyle{abbrv}
\bibliography{Bibliography,library}

\end{document}